\documentclass[10pt]{article}

\usepackage{bm}
\usepackage{fullpage}  

\usepackage{url}
\usepackage[ruled,linesnumbered]{algorithm2e}
\usepackage{multirow} 
\usepackage{hhline}  
\usepackage{pst-node}
\usepackage{mathtools,tikz-cd} 
\usepackage{tikz}
\usetikzlibrary{positioning} 
\usepackage{pst-plot}
\usepackage{verbatim}
\usepackage{algorithmic}
\usepackage{graphicx}
\usepackage{amsmath,amssymb,amsfonts,graphicx,amsthm,mathtools,nicefrac}
\usepackage{lscape}
\usepackage{color}
\usepackage{authblk}
\usepackage{subfigure}

\allowdisplaybreaks

\newtheorem{definition}{Definition}[section]
\newtheorem{theorem}{Theorem}[section]

\newtheorem{remark}{Remark}[section]

\numberwithin{equation}{section}

\newtheorem{lemma}[theorem]{Lemma}

\newtheorem{proposition}{Proposition}[section]

\DeclareMathOperator{\rank}{\mathrm{rank}}

\DeclareMathOperator{\trace}{trace}

\DeclareMathOperator{\vect}{vec}

\DeclareMathOperator{\argmin}{argmin}
\DeclareMathOperator{\TT}{TT}
\DeclareMathOperator{\gl}{GL}

\DeclareMathOperator{\grad}{grad}
\DeclareMathOperator{\Hess}{Hess}

\DeclareMathOperator{\ten}{ten}
\DeclareMathOperator{\QR}{QR}
\DeclareMathOperator{\TTSVD}{TT-SVD}
\DeclareMathOperator{\OS}{OS}

\def\lb{\left(}
\def\rb{\right)}

\def\lsb{\left[}
\def\rsb{\right]}

\newcommand{\matsnorm}[2]{\left\| #1\right\|_{{#2}}}

\newcommand{\fronorm}[1]{\ensuremath{\matsnorm{#1}{\footnotesize{\mathsf{F}}}}}

\newcommand{\bfm}[1]{\bm{#1}}

   \def\mA{\bfm A}  
\def\vb{\bfm b}   \def\mB{\bfm B}  
     
   \def\mD{\bfm D}  
\def\ve{\bfm e}     
    
   \def\mG{\bfm G}  
   \def\mH{\bfm H}  
   \def\mI{\bfm I}  
   \def\mJ{\bfm J}  
     
   \def\mL{\bfm L}  
   \def\mM{\bfm M}

   \def\mQ{\bfm Q}  
\def\vr{\bfm r}   \def\mR{\bfm R}  \def\R{\mathbb{R}}
   \def\mS{\bfm S}  
     
   \def\mU{\bfm U}  
   \def\mV{\bfm V}  
   \def\mW{\bfm W}  
\def\vx{\bfm x}   \def\mX{\bfm X}  
   \def\mY{\bfm Y}  
   \def\mZ{\bfm Z}

\def\calA{{\cal  A}} 
\def\calB{{\cal  B}}

\def\calG{{\cal  G}} 
\def\calH{{\cal  H}}

\def\calM{{\cal  M}}

\def\calP{{\cal  P}}

\def\calS{{\cal  S}} 
\def\calT{{\cal  T}} 
 
\def\calV{{\cal  V}} 
 
\def\calX{{\cal  X}} 
\def\calY{{\cal  Y}} 
\def\calZ{{\cal  Z}}

\def \tran {\mathsf{T}}

\def \ProjX1 {\calP_{\mX_1}^{(1)}}
\def \ProjX2 {\calP_{\mX_2}^{(2)}}
\def \ProjX3 {\calP_{\mX_3}^{(3)}}

\begin{document}

\title{Tensor Completion via  Tensor Train Based Low-Rank Quotient Geometry under a Preconditioned Metric\footnotetext{Authors  are listed alphabetically.}}

\author[1]{Jian-Feng Cai}
\author[2]{Wen Huang}
\author[3]{Haifeng Wang}
\author[3]{Ke Wei}
\affil[1]{Department of Mathematics, Hong Kong University of Science and Technology, Clear Water Bay, Kowloon, Hong Kong SAR, China.\vspace{.15cm}}
\affil[2]{School of Mathematical Sciences, Xiamen University, Xiamen, China.\vspace{.15cm}}
\affil[3]{School of Data Science, Fudan University, Shanghai, China.\vspace{.15cm}}

\date{}

\maketitle

\begin{abstract}
Low-rank tensor completion problem is about recovering a tensor from partially observed entries. 
We consider this problem in the tensor train format and extend the preconditioned metric from the matrix case to the tensor case. The first-order and second-order quotient geometry of the manifold of fixed tensor train rank tensors under this metric is studied in detail. Algorithms, including Riemannian gradient descent, Riemannian conjugate gradient,  and Riemannian Gauss-Newton, have been proposed for the tensor completion problem based on the quotient geometry. It has also been shown that the Riemannian Gauss-Newton method on the quotient geometry is equivalent to the Riemannian Gauss-Newton method on the embedded geometry with a specific retraction. 
Empirical evaluations  on random instances as well as on function-related tensors show that the proposed algorithms are competitive with other existing algorithms in terms of recovery ability, convergence performance, and reconstruction quality.\\

\noindent
\textbf{Keywords.} Low-rank tensor completion, tensor train decomposition,  Riemannian optimization, quotient geometry,  preconditioned metric
\end{abstract}

\section{Introduction}
Tensors are multidimensional arrays which arise in a wide range of applications, including but not limited to topic modeling~\cite{anandkumar2015spectral}, computer version~\cite{liu2012tensor}, collaborative filtering~\cite{karatzoglou2010multiverse},  and signal processing~\cite{cichocki2015tensor}. Tensor completion refers to the problem of recovering  the target tensor from its partial entries. It is not hard to see that, without any additional assumptions, tensor completion is an ill-posed problem. On the other hand,
this problem can be solved  when the target tensor 
 possesses certain intrinsic low-dimensional structures.  A notable example is low-rank tensor completion, where the target tensor is assumed to be low rank. In contrast to the matrix case, tensor has more complex rank notions up to different tensor decompositions such as
 CP decomposition~\cite{hitchcock1927expression}, Tucker decomposition~\cite{tucker1966some}, tensor train (TT) decomposition~\cite{oseledets2011tensor} (also  known in the computational physics community as matrix product state (MPS)~\cite{schollwock2011density,verstraete2008matrix}), and hierarchical Tucker (HT) decomposition~\cite{hackbusch2009new}.
 In this manuscript, we focus on the TT decomposition, a special form of the HT decomposition. 
Then the low-rank tensor completion problem can be formulated as follows:
\begin{align}\label{eq: original problem}
\min_{\calX\in\R^{n_1\times n_2\times\cdots\times n_d}}~h\lb\calX\rb := \frac{1}{2}\fronorm{\calP_{\Omega}\lb \calX\rb -\calP_{\Omega}\lb\calT\rb}^2\quad\text{s.t. }~\rank_{\TT}\lb\calX\rb = \vr,
\end{align}
where $\calT$ is the target tensor  to be recovered, $\Omega$ is a subset of indices for the observed entries, $\rank_{\TT}\lb\calX\rb$ is the TT rank of $\calX$ which will be introduced later, and $\calP_{\Omega}$ is the sampling operator defined by
\begin{align*}
\calP_{\Omega}\lb\calX\rb\lb i_1, \cdots, i_d\rb   &= \begin{cases}
		\calX\lb i_1,\cdots,i_d\rb,~&\text{if }(i_1,\cdots,i_d)\in\Omega,\\
		0, ~&\text{otherwise}.
\end{cases}
\end{align*}

\subsection{Preliminaries on Tensor Train Decomposition}\label{sec: preliminary}

\paragraph{Tensor train decomposition.} 
In the tensor train (TT) decomposition, a $d$-dimensional tensor 
$\calX\in\R^{n_1\times n_2\times\cdots\times n_d}$ can be  expressed  as a product of $d$ third order tensors. More precisely, the $\lb i_1,\cdots,i_d\rb$-th element of $\calX$ is 
\begin{align*}
    \calX\lb i_1,\cdots,i_d\rb = \sum_{\ell_1 = 1}^{r_1}\cdots\sum_{\ell_{d-1} = 1}^{r_{d-1}}\calX^1\lb 1,i_1,\ell_1\rb\calX^2\lb \ell_1,i_2,\ell_2\rb\cdots\calX^d\lb\ell_{d-1}, i_d,1\rb,
\end{align*}
where  $\calX^k\in\R^{r_{k-1}\times n_k\times r_k}$ are the core tensors, $k = 1,\cdots,d$, and $r_0 = r_d = 1$. For conciseness,  we denote by $\calX^k\lb i_k\rb \in\R^{r_{k-1}\times r_k}$ the $i_k$-th slice of $\calX^k$ which yields the following equivalent expression
\begin{align}\label{eq: tt decomposition}
    \calX\lb i_1,\cdots,i_d\rb = \calX^1\lb i_1\rb\calX^2\lb i_2\rb\cdots\calX^d\lb i_d\rb.
\end{align}
In addition, for any $\calZ\in\R^{n_1\times n_2\times n_3}$, the left and  right unfolding operators: $L: \R^{n_1\times n_2\times n_3}\rightarrow \R^{n_1n_2\times n_3}$, $R: \R^{n_1\times n_2\times n_3}\rightarrow \R^{n_1\times n_2n_3}$ are defined as
\begin{align*}
    L\lb \calZ\rb\in\R^{n_1n_2\times n_3}: L\lb \calZ\rb\lb i_1+n_1\lb i_2-1\rb, i_3\rb &= \calZ\lb i_1,i_2,i_3\rb,\\
    R\lb \calZ\rb\in\R^{n_1\times n_2n_3}: R\lb\calZ\rb\lb i_1, i_2+n_2\lb i_3-1\rb\rb  &= \calZ\lb i_1,i_2,i_3\rb.
\end{align*}
\paragraph{$k$-th unfolding and interface matrices.}
The $k$-th unfolding of a tensor $\calX$ is a matrix of size $n_1\cdots n_k\times n_{k+1}\cdots n_d$, defined by
\begin{align*}
 \calX^{<k>}\lb i_1,\cdots,i_k; i_{k+1},\cdots,i_d\rb = \calX\lb i_1,\cdots,i_d\rb,
\end{align*}
where the semicolon represents the separation of the row and column indices: the first $k$ indices of $\calX$ enumerate the rows of $\calX^{<k>}$, and the last $d-k$ the columns of $\calX^{<k>}$. Additionally, a tensor $\calX$ with core tensors $\left\lbrace \calX^1,\cdots,\calX^d\right\rbrace$ can be split into left and right parts
\begin{align*}
    \calX^{\leq k}\in\R^{n_1n_2\cdots n_k\times r_k}: \calX^{\leq k}\lb i_1,i_2,\cdots,i_k;:\rb &= \calX^1\lb i_1\rb\calX^2\lb i_2\rb\cdots\calX^k\lb i_k\rb,\\
    \calX^{\geq k}\in\R^{n_kn_{k+1}\cdots n_d\times r_{k-1}}:\calX^{\geq k}\lb i_k,i_{k+1},\cdots,i_d;:\rb &= \lsb\calX^k\lb i_k\rb\calX^{k+1}\lb i_{k+1}\rb\cdots\calX^d\lb i_d\rb\rsb^\tran,
\end{align*}
the so-called interface matrices~\cite{steinlechner2016riemannian}. The following recursive relations between the interface matrices, which follow immediately from the definition, will be very useful: for $k = 1,\cdots,d$,
\begin{align}\label{eq: recursive relation}
    \calX^{\leq k} = \lb\mI_{n_k}\otimes \calX^{\leq k-1}\rb L\lb\calX^k\rb,\notag\\
    \calX^{\geq k} = \lb \calX^{\geq k+1}\otimes\mI_{n_k} \rb R\lb\calX^k\rb^\tran,
\end{align}
where $\mI_{n_k}$ is the identity matrix of size $n_k\times n_k$, $\otimes$ denotes the Kronecker product and $\calX^{\leq 0} = \calX^{\geq d+1} := 1$. With these notations,  the $k$-th unfolding of $\calX$ can be expressed as 
\begin{align}\label{eq: kth unfolding}
    \calX^{<k>} = \calX^{\leq k} {\calX^{\geq k+1}}^\tran = \lb\mI_{n_k}\otimes \calX^{\leq k-1}\rb L\lb\calX^k\rb{\calX^{\geq k+1}}^\tran = \calX^{\leq k}R\lb\calX^{k+1}\rb\lb {\calX^{\geq k+2}}^\tran\otimes\mI_{n_{k+1}} \rb .
\end{align}
We also define the interface matrices product as follows
\begin{align}\label{eq: interface matrix product}
    \mL^k &= {\calX^{\leq k}}^\tran \calX^{\leq k}\in\R^{r_k\times r_k},\\
    \mR^k & = {\calX^{\geq k}}^\tran\calX^{\geq k}\in\R^{r_{k-1}\times r_{k-1}}.
\end{align}

\paragraph{TT rank.}
The TT rank of a tensor $\calX$ is defined as the smallest $\lb 1,r_1,\cdots,r_{d-1},1\rb$ such that $\calX$ admits a TT decomposition~\eqref{eq: tt decomposition} with core tensors of size $r_{k-1}\times n_k\times r_k$, for $k = 1,\cdots,d$. 
  The TT rank $r_k$  is closely related to the rank of the $k$-th unfolding of  $\calX$.
More precisely, if a tensor can be decomposed as~\eqref{eq: tt decomposition}, it necessarily holds that  $r_k\geq \rank\lb\calX^{<k>}\rb$ \cite{uschmajew2020geometric}.  Furthermore, there exists a TT decomposition with $r_k = \rank\lb\calX^{<k>}\rb$. Consequently,
  $r_k$ is equal to the rank of $\calX^{<k>}$ and 
\begin{align*}
  \rank_{\TT}\lb\calX\rb = \vr = \lb 1,r_1,\cdots,r_{d-1},1\rb := \lb 1,\rank\lb\calX^{<1>}\rb,\cdots,\rank\lb\calX^{<d-1>}\rb,1\rb.   
\end{align*}

In addition to the basics of TT decomposition, we also need the following two notions in this paper:
\begin{itemize}
    \item The mode-$k$ product of a tensor $\calX\in\R^{n_1\times n_2\times\cdots\times n_d}$ with a matrix $\mA\in\R^{m_k\times n_k}$, denoted $\calX\times_k\mA$, yields a tensor of size $n_1\times\cdots\times m_k\times\cdots\times n_d$. The $\lb i_1,\cdots,j_k,\cdots,i_d\rb$-th entry of $\calX\times_k\mA$ is
\begin{align*}
    \lb\calX\times_k\mA\rb \lb i_1,\cdots,j_k,\cdots,i_d\rb = \sum_{i_k = 1}^{n_k}\calX\lb i_1,\cdots,i_k,\cdots,i_d\rb\mA\lb j_k,i_k\rb.
\end{align*}
\item The matricization operator for a tensor $\calX\in\R^{n_1\times\cdots\times n_d}$  is defined as
\begin{align*}
    \calM_k\lb\calX\rb\in\R^{n_k\times \prod_{j\neq k}n_j}:~\calM_k\lb\calX\rb\lb i_k,1+\sum_{\ell = 1,\ell\neq k}^n\lb i_{\ell}-1\rb J_{\ell}\rb = \calX\lb i_1,\cdots,i_d\rb,
\end{align*}
where $J_\ell = \prod_{j = 1,j\neq k}^{\ell-1}n_j$. It can be verified that we have the following two equations related to~\eqref{eq: kth unfolding}: for $k = 1,\cdots,d$,
\begin{align}\label{eq: matricization}
    \calM_3\lb \calX^k\times_1\calX^{\leq k-1}\times_3\calX^{\geq k+1}\rb &= \lb\lb\mI_{n_k}\otimes \calX^{\leq k-1}\rb L\lb\calX^k\rb{\calX^{\geq k+1}}^\tran\rb^\tran,\\
    \calM_2\lb \calX^k\times_1\calX^{\leq k-1}\times_3\calX^{\geq k+1}\rb &= \calM_2\lb\calX^k\rb\lb \calX^{\geq k+1}\otimes\calX^{\leq k-1}\rb^\tran.\notag
\end{align}
\end{itemize}


\subsection{Geometric Structure}\label{sec: gs}
Let $\calM_{\vr}$ be a set of fixed  tensor train rank tensors, that is
\begin{align}\label{eq: embedded manifold}
    \calM_{\vr} = \left\lbrace \calX\in\R^{n_1\times n_2\times\cdots\times n_d}:\rank_{\TT}\lb\calX\rb = \vr\right\rbrace.
\end{align}
For $\calM_{\vr}$ to be non-empty, the   necessary and sufficient conditions are
\begin{align*}
    r_{k-1}\leq n_kr_k, ~ r_k\leq n_kr_{k-1},~k = 1,\cdots,d,
\end{align*}
see for example~\cite[Section 9.3.3]{uschmajew2020geometric}.
 Moreover, it has been shown that the set $\calM_{\vr}$
forms a smooth embedded submanifold of dimension $\sum_{k= 1}^d r_{k-1}n_kr_k-\sum_{k = 1}^{d-1}r_k^2$~\cite{holtz2012manifolds}. 

To introduce the quotient geometry, 
assume $\calX$ is represented in the TT format~\eqref{eq: tt decomposition} 
 with core tensors $\bar{\calX}:=\left\lbrace \calX^1,\cdots,\calX^d\right\rbrace$. It is known that the condition $\rank_{\TT}\lb\calX\rb = \vr$ is equivalent to 
 $\rank\lb L\lb\calX^k\rb\rb = r_k$ and $\rank\lb R\lb\calX^k\rb\rb = r_{k-1}$, $k = 1,\cdots,d$~\cite{holtz2012manifolds,uschmajew2020geometric}.    We denote by $\overline{\calM}_{\vr} := \R_{\ast}^{r_0\times n_1\times r_1}\times\cdots\times\R_{\ast}^{r_{d-1}\times n_d\times r_d}$ the set of tensors with the following rank constraint: for any $\calX^k\in \R_{\ast}^{r_{k-1}\times n_k\times r_{k}}$, $\rank\lb L\lb\calX^k\rb\rb = r_k$, and $\rank\lb R\lb\calX^k\rb\rb = r_{k-1}$, $k = 1,\cdots,d$. Let the mapping $\phi$ be
\begin{align*}
    \phi: \overline{\calM}_{\vr}\rightarrow\calM_{\vr} : \bar{\calX} \rightarrow  \phi\lb\bar{\calX}\rb, \hbox{ such that } \phi\lb\bar{\calX}\rb\lb i_1,\cdots,i_d\rb = \calX^1\lb i_1\rb\calX^2\lb i_2\rb\cdots\calX^d\lb i_d\rb.
\end{align*}
It is not hard to see that the image of $\overline{\calM}_{\vr}$ under $\phi$ is $\calM_{\vr}$. 
     Let $\bar{\calX} = \left\lbrace \calX^1,\cdots,\calX^d\right\rbrace\in\overline{\calM}_{\vr}$ and $\bar{\calY} = \left\lbrace \calY^1,\cdots,\calY^d\right\rbrace\in\overline{\calM}_{\vr}$. Then by Proposition 7 in ~\cite{uschmajew2013geometry}, 
 $\phi\lb\bar{\calX}\rb = \phi\lb\bar{\calY}\rb$ if and only if there exist invertible matrices $\mA_1,\cdots,\mA_{d-1}$ of appropriate sizes such that
\begin{align}\label{eq: barX equivalent to barY}
    \bar{\calY}  = \left\lbrace\calX^1\times_3\mA_1^\tran,\calX^2\times_1\mA_1^{-1}\times_3\mA_2^\tran,\cdots,\calX^d\times_1\mA_{d-1}^{-1}\right\rbrace.
\end{align}
 We denote by  $\lsb\bar{\calX}\rsb$ the set  containing all points $\bar{\calX}\in\overline{\calM}_{\vr}$ that obeys $\phi\lb\bar{\calX}\rb = \calX$. This set is known as the equivalent class. Moreover, let $\calG$ be the Lie group
\begin{align*}
\calG = \left\{\mathrm{A} = \lb\mA_1,\cdots,\mA_{d-1}\rb:  \mA_k\in \gl(r_k),~\text{for}~ k = 1,\cdots,d-1 \right\},
\end{align*}
where $\gl(r_k)$ is the set of non-singular matrices of size $r_k\times r_k$, and define the quotient set
\begin{align}\label{eq: quotient manifold}
    \overline{\calM}_{\vr}/\calG &= \left\lbrace  \lsb\bar{\calX}\rsb: \bar{\calX}\in\overline{\calM}_{\vr}\right\rbrace.
\end{align}
As shown in~\cite{uschmajew2013geometry},  $\overline{\calM}_{\vr}/\calG$ is  a quotient manifold of dimension $\sum_{k= 1}^d r_{k-1}n_kr_k-\sum_{k = 1}^{d-1}r_k^2$.

The  natural projection $\pi$ which maps the element in the total space $\overline{\calM}_{\vr}$ to the quotient space $\overline{\calM}_{\vr}/\calG$ is defined as
\begin{align*}
    \pi: \overline{\calM}_{\vr}\rightarrow \overline{\calM}_{\vr}/\calG: \bar{\calX}\rightarrow \pi\lb\bar{\calX}\rb= \lsb\bar{\calX}\rsb.
\end{align*}
Then, it is evident that there is a bijective mapping $\Phi$ from   $\overline{\calM}_{\vr}/\calG$ to $\calM_{\vr}$ such that $\phi = \Phi\circ\pi$.
 The relations between the different geometric spaces are  summarized in the diagram below: 
 \begin{equation*}
  \begin{tikzpicture}{[baseline=\the\dimexpr\fontdimen22\textfont2\relax]
   \node      (quotient)                             {$\overline{\calM}_{\vr}/\calG$};
   \node      (total)  [above= of quotient]                             {$\overline{\calM}_{\vr}$};
    \node      (embedded)   [right= of quotient]                           {$\calM_{\vr}$};
  \draw[->] (total.south) -- node[right] {$\pi$} (quotient.north) ;
  \draw[->] (quotient.east) -- node[above] {$\Phi$}(embedded.west) ;
  \draw[->] (total.south east) -- node[right] {$~\phi:= \Phi\circ \pi$} (embedded.north west);
  }
  \end{tikzpicture}
\end{equation*}

\subsection{Main Contributions and Outline}
  For the low rank tensor completion problem in the tensor train format, several computational methods have been developed, including block coordinate descent~\cite{bengua2017efficient}, iterative hard thresholding~\cite{rauhut2017low}, gradient-based optimization~\cite{yuan2017completion}, Riemannian optimization~\cite{cai2022provable,steinlechner2016riemannian,wang2019tensor}. In this paper, we study this problem based on the quotient geometry under a specific metric.  The main contributions of this paper are summarized as follows:
\begin{itemize}
    \item 
    We extend the preconditioned metric from matrix  to tensor  and 
     exploit the first order and second order geometry of the quotient manifold $\overline{\calM}_{\vr}/\calG$ under this metric.  Even though the results are extensions from the matrix case,  the mathematical derivations are by no means trivial due to the complications of the tensor algebra. In particular, to compute the projection onto the horizontal space, we  have to solve a system of linear equations where the coefficient matrix is  symmetric and block tridiagonal. A fundamental contribution of this paper is  that  the positive definiteness of the  coefficient matrix has been established. 
    \item Riemannian optimization algorithms based on the quotient geometry are proposed, including  Riemannian gradient descent, Riemannian conjugate gradient, and  Riemannian Gauss-Newton.
    In particular, it has  been shown that the Riemannian Gauss-Newton method on the quotient geometry is equivalent to the Riemannian Gauss-Newton method on the embedded geometry with a specific retraction.
      The per iteration computational complexity of  the first order algorithms presented in this paper scales linearly in the dimension $d$, in the tensor size $n$ and in the sampling set size $|\Omega|$,  scales polynomially in the TT rank $\vr$. Overall, it is comparable with that  of the Riemannian conjugate gradient method proposed in \cite{steinlechner2016riemannian} based on the submanifold. Numerical experiments  demonstrate that the proposed algorithms for the tensor completion problem are competitive with other state-of-the-art algorithms in terms of recovery ability, convergence performance, and reconstruction quality.
\end{itemize}

The rest of this manuscript is outlined as follows. Section $2$ investigates the first order and second order geometry of $\overline{\calM}_{\vr}/\calG$ under the preconditioned metric. Riemannian gradient descent, Riemannian conjugate gradient, and  Riemannian Gauss-Newton methods are presented in Section $3$. 
Empirical performance evaluations of the algorithms are given in Section $4$. In Section $5$,  we conclude this paper with  some future research directions.

\section{Quotient Geometry under Preconditioned Metric}

\subsection{Preconditioned Metric and Horizontal Lift}

Recall that the total space $\overline{\calM}_{\vr}$ is defined as $\overline{\calM}_{\vr} = \R_{\ast}^{r_0\times n_1\times r_1}\times\cdots\times\R_{\ast}^{r_{d-1}\times n_d\times r_d}$.
The vertical space, denoted by $\calV_{\bar{\calX}}$,  is the tangent space  to the equivalent class $\lsb\bar{\calX}\rsb$ at $\bar{\calX}$. The expression for $\calV_{\bar{\calX}}$ is given in the following proposition.
\begin{proposition}\label{pro: vertical space}
The vertical space at $\bar{\calX}$ is
\begin{align*}
\calV_{\bar{\calX}} = \left\lbrace \left\lbrace \calX^1\times_3\mD_1^\tran,-\calX^2\times_1\mD_1+\calX^2\times_3\mD_2^\tran,\cdots,-\calX^d\times_1\mD_{d-1}\right\rbrace:  \mD_k\in\R^{r_k\times r_k}, k = 1,\cdots, d-1  \right\rbrace.
\end{align*}
\end{proposition}
\begin{proof}
The proof follows from~\cite[Section 4.3]{uschmajew2013geometry}.
\end{proof}
Notice that the vertical space is a subspace of $T_{\bar{\calX}}\overline{\calM}_{\vr}$. The horizontal space, denoted by $\calH_{\bar{\calX}}$, is any  subspace of $T_{\bar{\calX}}\overline{\calM}_{\vr}$ that is complementary to $\calV_{\bar{\calX}}$. For the quotient manifold $\overline{\calM}_{\vr}/\calG$~\eqref{eq: quotient manifold}, any 
element $\bar{\xi}\in T_{\bar{\calX}}\overline{\calM}_{\vr}$  satisfying $\mathrm{D}\pi\lb\bar{\calX}\rb\lsb \bar{\xi}\rsb = \xi_{\lsb\bar{\calX}\rsb}$ can be seen as a representation of $\xi_{\lsb\bar{\calX}\rsb}$ where $\xi_{\lsb\bar{\calX}\rsb}\in T_{\lsb\bar{\calX} \rsb}\overline{\calM}_{\vr}/\calG$. Since the kernel of $\mathrm{D}\pi\lb\bar{\calX}\rb: T_{\bar{\calX}}\overline{\calM}_{\vr}\rightarrow T_{\lsb\bar{\calX} \rsb}\overline{\calM}_{\vr}/\calG$ is the vertical space $\calV_{\bar{\calX}}$, there are  infinitely many representations of $\xi_{\lsb\bar{\calX}\rsb}$ in $T_{\bar{\calX}}\overline{\calM}_{\vr}$.  Nevertheless, 
one can find a unique representation of $\xi_{\lsb\bar{\calX}\rsb}$ in horizontal space $\calH_{\bar{\calX}}$.   The tangent vector $\bar{\xi}\in \calH_{\bar{\calX}}$ satisfying $\mathrm{D}\pi\lb\bar{\calX}\rb\lsb \bar{\xi}\rsb = \xi_{\lsb\bar{\calX}\rsb}$ is called the horizontal lift of $\xi_{\lsb\bar{\calX}\rsb}$ at $\bar{\calX}$. Throughout this paper, the horizontal lift of $\xi_{\lsb\bar{\calX}\rsb}\in T_{\lsb \bar{\calX}\rsb}\overline{\calM}_{\vr}$ at $\bar{\calX}$ is denoted by $\bar{\xi}_{\bar{\calX}}$.

Regarding the horizontal space, a particular one by imposing orthogonal conditions is proposed in~\cite{uschmajew2013geometry}. This horizontal space has also been exploited  in~\cite{da2015optimization} for the development of Riemannian quotient algorithms. Moreover, given a Riemannian metric $\bar{g} \lb\cdot,\cdot\rb$, 
one can construct the horizontal space $\calH_{\bar{\calX}}$ which is  orthogonal complementary to the vertical space $\calV_{\bar{\calX}}$:
\begin{align}\label{eq: def of horizontal space}
    \calH_{\bar{\calX}} = \left\lbrace \bar{\xi}\in T_{\bar{\calX}}\overline{\calM}_{\vr}: \bar{g}_{\bar{\calX}}\lb \bar{\xi},\bar{\eta}\rb = 0,~\text{for all}~\bar{\eta}\in\calV_{\bar{\calX}}\right\rbrace.
\end{align}
In this manuscript, we consider a preconditioned metric which is extended from the matrix case~\cite{dong2022analysis,kasai2016low,mishra2012riemannian,mishra2016riemannian,zheng2022riemannian}. Let $\R_{\ast}^{n_1\times r}\times \R_{\ast}^{n_2\times r}$ be the space of matrices with full column rank $r$. Given $\bar{x} := \left\lbrace \mG,\mH\right\rbrace\in \R_{\ast}^{n_1\times r}\times \R_{\ast}^{n_2\times r}$, the preconditioned metric on the tangent space of $\R_{\ast}^{n_1\times r}\times \R_{\ast}^{n_2\times r}$ at $\bar{x}$ is defined as
\begin{align}\label{eq: premetric matrix}
    \bar{g}_{\bar{x}}\lb\bar{\xi},\bar{\eta}\rb  = \trace\lb{\xi^1}^\tran\eta^1\lb\mH^\tran\mH\rb\rb+\trace\lb {\xi^2}^\tran\eta^2\lb\mG^\tran\mG\rb\rb,
\end{align}
where $\bar{\xi} = \left\lbrace \xi^1,\xi^2\right\rbrace$ and $\bar{\eta} = \left\lbrace \eta^1,\eta^2\right\rbrace\in T_{\bar{x}}\R_{\ast}^{n_1\times r}\times \R_{\ast}^{n_2\times r}$. Under this metric, the Riemannian gradient of a function $\bar{f}$ at $\bar{x}$ is given by
\begin{align}\label{eq: riegrad general}
    \grad \bar{f}(\bar{x}) = \lb\nabla_{\mG}\bar{f}(\bar{x}) \lb\mH^\tran\mH\rb^{-1},\nabla_{\mH}\bar{f}(\bar{x}) \lb\mG^\tran\mG\rb^{-1} \rb,
\end{align}
where $\nabla_{\mG}\bar{f}(\bar{x})$ and $\nabla_{\mH}\bar{f}(\bar{x})$ are the Euclidean partial derivatives of $\bar{f}$. Equivalently, we rewrite~\eqref{eq: riegrad general} in a vectorization form as
\begin{align*}
    \vect\lb     \grad \bar{f}(\bar{x}) \rb &= \underbrace{\begin{bmatrix}
        \lb\mH^\tran\mH\rb^{-1}\otimes\mI_{n_1} &\bm{0}\\
        \bm{0} & \lb\mG^\tran\mG\rb^{-1}\otimes\mI_{n_2}
    \end{bmatrix}}_{=:\mJ^{-1}}\vect\lb \lb\nabla_{\mG}\bar{f}(\bar{x}),\nabla_{\mH}\bar{f}(\bar{x})\rb\rb.
\end{align*}
Thus,  the Riemannian gradient descent direction~\eqref{eq: riegrad general} can be viewed as an approximation of the Newton direction. The metric in~\eqref{eq: premetric matrix} is known as the preconditioned metric on $\R_{\ast}^{n_1\times r}\times \R_{\ast}^{n_2\times r}$.

 Note that it is not evident to extend the preconditioned metric from the form presented in \eqref{eq: premetric matrix} because there are tensor factors in the TT format. However, the following equivalent expression of \eqref{eq: premetric matrix} provides a more convenient form for the extension:
\begin{align*}
\bar{g}_{\bar{x}}\lb\bar{\xi},\bar{\eta}\rb = \left\langle \xi^1\mH^\tran,\eta^1\mH^\tran\right\rangle+\left\langle \mG{\xi^2}^\tran,\mG{\eta^2}^\tran\right\rangle.
\end{align*}
Basically, the precondition metric is given by replacing each factors in $\left\langle \mG\mH^\tran,\mG\mH^\tran\right\rangle$ by the corresponding tangent vectors in the same mode. This observation leads to the following generalization in the tensor case.
\begin{definition}
    Given $\bar{\calX} = \left\lbrace\calX^1,\cdots,\calX^d\right\rbrace\in \overline{\calM}_{\vr}$, the preconditioned metric $\bar{g}_{\bar{\calX}}: T_{\bar{\calX}}\overline{\calM}_{\vr}\times T_{\bar{\calX}}\overline{\calM}_{\vr}$ is defined as follows:
\begin{align}\label{eq: premetric}
    \bar{g}_{\bar{\calX}}\lb \bar{\xi},\bar{\eta}\rb &= \sum_{k = 1}^d\left\langle \phi\lb \left\lbrace \calX^1,\cdots,\calX^{k-1},\xi^k,\calX^{k+1},\cdots,\calX^d\right\rbrace\rb, \phi\lb \left\lbrace \calX^1,\cdots,\calX^{k-1},\eta^k,\calX^{k+1},\cdots,\calX^d\right\rbrace\rb \right\rangle,
\end{align}
    where $\bar{\xi}:= \left\lbrace \xi^k\right\rbrace_{k = 1}^d,\bar{\eta}:= \left\lbrace \eta^k\right\rbrace_{k = 1}^d \in T_{\bar{\calX}}\overline{\calM}_{\vr}$.
\end{definition}
\begin{lemma}
    The preconditioned metric defined in~\eqref{eq: premetric} has  the following equivalent expression:
    \begin{align*}
        \bar{g}_{\bar{\calX}}\lb \bar{\xi},\bar{\eta}\rb = \sum_{k = 1}^d \trace\lb\calM_2\lb\xi^k\rb^\tran\calM_{2}\lb\eta^k\rb\lb{\mH^k}^\tran\mH^k\rb\rb,
    \end{align*}
    which coincides with  that in~\eqref{eq: premetric matrix}. Here  $\mH^k = \calX^{\geq k+1}\otimes\calX^{\leq k-1}$, for $k = 1,\cdots,d$.
\end{lemma}
\begin{proof}
    The application of~\eqref{eq: kth unfolding} yields that
    \begin{align*}
    \bar{g}_{\bar{\calX}}\lb \bar{\xi},\bar{\eta}\rb &= \sum_{k = 1}^d\left\langle \phi\lb \left\lbrace \calX^1,\cdots,\calX^{k-1},\xi^k,\calX^{k+1},\cdots,\calX^d\right\rbrace\rb, \phi\lb \left\lbrace \calX^1,\cdots,\calX^{k-1},\eta^k,\calX^{k+1},\cdots,\calX^d\right\rbrace\rb \right\rangle,\\
 & = \sum_{k = 1}^d\left\langle 
    \lb \mI_{n_k}\otimes \calX^{\leq k-1}\rb L\lb\xi^k\rb{\calX^{\geq k+1}}^\tran,
    \lb \mI_{n_k}\otimes \calX^{\leq k-1}\rb L\lb\eta^k\rb{\calX^{\geq k+1}}^\tran
 \right\rangle\notag\\
 & = \sum_{k = 1}^d\left\langle  \xi^k\times_1\calX^{\leq k-1}\times_3\calX^{\geq k+1},\eta^k\times_1\calX^{\leq k-1}\times_3\calX^{\geq k+1}\right\rangle\\
 & = \sum_{k = 1}^d\left\langle  \calM_2\lb\xi^k\rb{\mH^k}^\tran, \calM_2\lb\eta^k\rb{\mH^k}^\tran\right\rangle\\
 & =  \sum_{k = 1}^d \trace\lb\calM_2\lb\xi^k\rb^\tran\calM_{2}\lb\eta^k\rb\lb{\mH^k}^\tran\mH^k\rb\rb,
\end{align*}
where the second and third equations follow from~\eqref{eq: matricization}.
\end{proof}
\begin{proposition}
    Under the preconditioned metric $\bar{g}_{\bar{\calX}}$ defined in~\eqref{eq: premetric},  the horizontal space at $\bar{\calX}$ is
    \begin{align*}
        \calH_{\bar{\calX}} &= \bigg\{  \bar{\xi}\in T_{\bar{\calX}}\overline{\calM}_{\vr}: L\lb\calX^k\rb^\tran\lb\mI_{n_k}\otimes \mL^{k-1}\rb L\lb\xi^{k}\rb\mR^{k+1} \\
        &\qquad\qquad\qquad \qquad=\mL^{k}R\lb\xi^{k+1}\rb\lb \mR^{k+2}\otimes\mI_{n_{k+1}} \rb R\lb\calX^{k+1}\rb^\tran, k = 1,\cdots,d-1\bigg\}.
    \end{align*}
\end{proposition}
\begin{proof}
    By the definition of the horizontal space in~\eqref{eq: def of horizontal space}, for  a tangent vector $\bar{\xi}\in \calH_{\bar{\calX}}$, we must have $\bar{g}_{\bar{\calX}}\lb\bar{\xi},\bar{\eta}\rb = 0$ for all $\bar{\eta}\in \calV_{\bar{\calX}}$. 
    With the application of Proposition~\ref{pro: vertical space},
    the equation $\bar{g}_{\bar{\calX}}\lb\bar{\xi},\bar{\eta}\rb = 0$ reduces to
    \begin{align*}
        0&=\sum_{k = 1}^d\left\langle \phi\lb \left\lbrace \calX^1,\cdots,\xi^k,\cdots,\calX^d\right\rbrace\rb, \phi\lb \left\lbrace \calX^1,\cdots,\eta^k,\cdots,\calX^d\right\rbrace\rb \right\rangle\\
        &  = \left\langle \phi\lb \left\lbrace \xi^1,\cdots,\calX^d\right\rbrace\rb, \phi\lb \left\lbrace \calX^1\times_3\mD_1^\tran,\cdots,\calX^d\right\rbrace\rb \right\rangle+\left\langle \phi\lb \left\lbrace \calX^1,\cdots,\xi^d\right\rbrace\rb, \phi\lb \left\lbrace \calX^1,\cdots,-\calX^d\times_1\mD_{d-1}\right\rbrace\rb \right\rangle \\
        &\quad+\sum_{k = 2}^{d-1}\left\langle \phi\lb \left\lbrace \calX^1,\cdots,\xi^k,\cdots,\calX^d\right\rbrace\rb, \phi\lb \left\lbrace \calX^1,\cdots,-\calX^k\times_1\mD_{k-1}+\calX^k\times_3\mD_k^\tran,\cdots,\calX^d\right\rbrace\rb \right\rangle \\
        & = \sum_{k =1}^{d-1} \left\langle \phi\lb \left\lbrace \calX^1,\cdots,\xi^{k},\cdots,\calX^d\right\rbrace\rb- \phi\lb \left\lbrace \calX^1,\cdots,\xi^{k+1},\cdots,\calX^d\right\rbrace\rb, \phi\lb \left\lbrace \calX^1,\cdots,\calX^{k}\times_3\mD_{k}^\tran,\cdots,\calX^d\right\rbrace\rb \right\rangle\\
        & = \sum_{k= 1}^{d-1} \left\langle  \lb\mI_{n_k}\otimes \calX^{\leq k-1}\rb L\lb\xi^{k}\rb{\calX^{\geq k+1}}^\tran -\calX^{\leq k}R\lb\xi^{k+1}\rb\lb {\calX^{\geq k+2}}^\tran\otimes\mI_{n_{k+1}}\rb ,  \calX^{\leq k} \mD_{k}{\calX^{\geq k+1}}^\tran\right\rangle\\
        & = \sum_{k= 1}^{d-1} \left\langle  {\calX^{\leq k}}^\tran\lb\lb\mI_{n_k}\otimes \calX^{\leq k-1}\rb L\lb\xi^{k}\rb{\calX^{\geq k+1}}^\tran -\calX^{\leq k}R\lb\xi^{k+1}\rb\lb {\calX^{\geq k+2}}^\tran\otimes\mI_{n_{k+1}}\rb\rb\calX^{\geq k+1} ,   \mD_{k}\right\rangle\\
        &=\sum_{k= 1}^{d-1} \left\langle  L\lb\calX^k\rb^\tran\lb\mI_{n_k}\otimes \mL^{k-1}\rb L\lb\xi^{k}\rb\mR^{k+1} -\mL^{k}R\lb\xi^{k+1}\rb\lb \mR^{k+2}\otimes\mI_{n_{k+1}} \rb R\lb\calX^{k+1}\rb^\tran ,   \mD_{k}\right\rangle
    \end{align*}
    where the fourth equation is due to~\eqref{eq: kth unfolding} and the last line follows from~\eqref{eq: recursive relation} and~\eqref{eq: interface matrix product}. Since $\mD_{k}\in\R^{r_k\times r_k}$ is an arbitrary matrix, one can conclude that for $k = 1, \cdots,d-1$,
    \begin{align}\label{eq: horizontal equation}      L\lb\calX^k\rb^\tran\lb\mI_{n_k}\otimes \mL^{k-1}\rb L\lb\xi^{k}\rb\mR^{k+1} =\mL^{k}R\lb\xi^{k+1}\rb\lb \mR^{k+2}\otimes\mI_{n_{k+1}} \rb R\lb\calX^{k+1}\rb^\tran,
    \end{align}
    which completes the proof.
\end{proof}

The projections of any $\bar{\xi} = \left\lbrace \xi^1,\cdots,\xi^d\right\rbrace\in T_{\bar{\calX}}\overline{\calM}_{\vr}$ onto the vertical and horizontal spaces, denoted $\calP_{\bar{\calX}}^{\calV}\lb\bar{\xi}\rb$ and $\calP_{\bar{\calX}}^{\calH}\lb\bar{\xi}\rb$,  are given by the following lemma.
\begin{lemma}
Under the preconditioned metric $\bar{g}_{\bar{\calX}}$ defined in~\eqref{eq: premetric},  the projections onto the vertical  and horizontal spaces are given by
\begin{align}\label{eq: horizontal projection}
\calP_{\bar{\calX}}^{\calV} \lb\bar{\xi}\rb &= \lb\calX^1\times_3\mD_1^\tran, -\calX^2\times_1\mD_1+\calX^2\times_3\mD_2^\tran,\cdots, -\calX^{d}\times_1\mD_{d-1}\rb,\notag\\
\calP_{\bar{\calX}}^{\calH} \lb\bar{\xi}\rb &= \lb  \xi^1-\calX^1\times_3\mD_1^\tran,\xi^2+\calX^2\times_1\mD_1-\calX^2\times_3\mD_2^\tran,\cdots,\xi^d+\calX^{d}\times_1 \mD_{d-1}\rb.
\end{align}
where $\mD_k\in\R^{r_k\times r_k}$ are uniquely determined by  the following system of linear equations
\begin{align}\label{eq: linear equation}
\begin{bmatrix}
\mA_1 &\mB_1\\
\mB_1^\tran &\mA_2&\mB_2\\
&\ddots&\ddots&\ddots\\
& &\mB_{d-2}^\tran&\mA_{d-1}
\end{bmatrix}
\begin{bmatrix}
\vect(\mD_1)\\
\vect(\mD_2)\\
\vdots\\
\vect(\mD_{d-1})
\end{bmatrix} = \begin{bmatrix}
\vb_1\\
\vb_2\\
\vdots\\
\vb_{d-1}
\end{bmatrix}.
\end{align}
Here, the matrices $\mA_k$,$\mB_k$ and $\vb_k$ are 
\begin{align*}
    \mA_k &= \mR^{k+1}\otimes\mL^k,\\
    \mB_k &= -\frac{1}{2}  \lb R\lb\calX^{k+1}\rb\otimes\mL^k\rb \lb\mR^{k+2}\otimes L\lb\calX^{k+1}\rb\rb,\\
    \vb_k&= \frac{1}{2} \lb\lb\mR^{k+1}\otimes L\lb\calX^k\rb^\tran \rb \vect\lb \xi^k\times_1 \mL^{k-1} \rb-\lb 
 R\lb\calX^{k+1}\rb\otimes\mL^k\rb\vect\lb \xi^{k+1}\times_3\mR^{k+2}\rb\rb.
\end{align*}
\end{lemma}
\begin{proof}
    Since $\calP_{\bar{\calX}}^{\calH} \lb\bar{\xi}\rb\in\calH_{\bar{\calX}}$, the elements of  $\calP_{\bar{\calX}}^{\calH} \lb\bar{\xi}\rb$ must satisfy the equation~\eqref{eq: horizontal equation}. Substituting the $k$-th core tensor of 
 $\calP_{\bar{\calX}}^{\calH} \lb\bar{\xi}\rb$ into~\eqref{eq: horizontal equation} yields that
\begin{align*}
\mL^k\mD_k\mR^{k+1} &= \frac{1}{2} L\lb\calX^k\rb^\tran\lb\mI_{n_k}\otimes\mL^{k-1}\rb L\lb\xi^k+\calX^k\times_1\mD_{k-1} \rb\mR^{k+1}\\
&\quad-\frac{1}{2}\mL^k R\lb\xi^{k+1}-\calX^{k+1}\times_3\mD_{k+1}^\tran\rb\lb\mR^{k+2}\otimes\mI_{n_{k+1}}\rb R\lb\calX^{k+1}\rb^\tran.
\end{align*}
Vectorizing both sides of this equation gives that
\begin{align*}
    \lb\mR^{k+1}\otimes\mL^k\rb\vect\lb\mD_k\rb &= \frac{1}{2}\lb\mR^{k+1}\otimes L\lb\calX^k\rb^\tran \rb\vect\lb \lb\mI_{n_k}\otimes\mL^{k-1}\rb L\lb\xi^k+\calX^k \times_1\mD_{k-1} \rb \rb\\
    &\quad-\frac{1}{2}\lb 
 R\lb\calX^{k+1}\rb\otimes\mL^k\rb\vect\lb R\lb\xi^{k+1}-\calX^{k+1}\times_3\mD_{k+1}^\tran\rb\lb\mR^{k+2}\otimes\mI_{n_{k+1}}\rb \rb\\
 &= \frac{1}{2}\lb\mR^{k+1}\otimes L\lb\calX^k\rb^\tran \rb\vect\lb \xi^k\times_1\mL^{k-1}+\calX^k\times_1\lb\mL^{k-1}\mD_{k-1}\rb \rb\\
 &\quad - \frac{1}{2}\lb 
 R\lb\calX^{k+1}\rb\otimes\mL^k\rb\vect\lb \xi^{k+1}\times_3\mR^{k+2}-\calX^{k+1}\times_3\lb\mR^{k+2}\mD_{k+1}^\tran\rb \rb\\
 & = \frac{1}{2}\lb\mR^{k+1}\otimes L\lb\calX^k\rb^\tran \rb\lb R\lb \calX^k\rb^\tran\otimes \mL^{k-1}\rb\vect\lb\mD_{k-1}\rb\\
 &\quad +\frac{1}{2}\lb R\lb\calX^{k+1}\rb\otimes\mL^k\rb \lb\mR^{k+2}\otimes L\lb\calX^{k+1}\rb\rb\vect\lb\mD_{k+1}\rb\\
 &\quad +\frac{1}{2} \lb\mR^{k+1}\otimes L\lb\calX^k\rb^\tran \rb \vect\lb  \xi^k\times_1\mL^{k-1}\rb\\
 &\quad-\frac{1}{2}\lb 
 R\lb\calX^{k+1}\rb\otimes\mL^k\rb\vect\lb \xi^{k+1}\times_3\mR^{k+2}\rb,
\end{align*}
where the second equation is due to~\eqref{eq: matricization}. Thus   one can obtain $d-1$  equations for $k = 1,\cdots,d-1$. Stacking them together yields the system of linear equations~\eqref{eq: linear equation}. The existence and uniqueness of projection implies the invertibility of the  coefficient matrix in~\eqref{eq: linear equation}. As a result, $\mD_k\in\R^{r_k\times r_k}$ can be uniquely obtained by solving~\eqref{eq: linear equation}.
\end{proof}
  Moreover, it can be  shown that the coefficient matrix in~\eqref{eq: linear equation} is positive definite.
\begin{lemma}\label{lem: positive definite}
     The symmetric block tridiagonal matrix  in~\eqref{eq: linear equation} is positive definite.
\end{lemma}
\begin{proof}
Since the symmetric block tridiagonal matrix in~\eqref{eq: linear equation} is invertible, we only need to verify the positive semidefiniteness of this matrix.
    For any $\vx = \lsb\vx_1^\tran,\cdots,\vx_{d-1}^\tran \rsb^\tran\in\R^{\sum_{k = 1}^{d-1}r_k^2}$ with $\vx_k\in\R^{r_k^2}$, $k = 1,\cdots,d-1$,  we have
\begin{align*}
\vx^\tran\underbrace{\begin{bmatrix}
\mA_1 &\mB_1\\
\mB_1^\tran &\mA_2&\mB_2\\
&\ddots&\ddots&\ddots\\
& &\mB_{d-2}^\tran&\mA_{d-1}
\end{bmatrix}}_{:= \mM}\vx
& = \begin{bmatrix}
\vx_1^\tran&\vx_2^\tran
\end{bmatrix}
\begin{bmatrix}
\mA_1 &\mB_1\\
\mB_1^\tran &\frac{1}{2}\mA_2
\end{bmatrix}
 \begin{bmatrix}
\vx_1\\
\vx_2
\end{bmatrix}
+
 \begin{bmatrix}
\vx_2^\tran&\vx_3^\tran
\end{bmatrix}
\begin{bmatrix}
\frac{1}{2}\mA_2 &\mB_2\\
\mB_2^\tran &\frac{1}{2}\mA_3
\end{bmatrix}
 \begin{bmatrix}
\vx_2\\
\vx_3
\end{bmatrix}\\
&\quad+\cdots
+  \begin{bmatrix}
\vx_{d-2}^\tran&\vx_{d-1}^\tran
\end{bmatrix}
\begin{bmatrix}
\frac{1}{2}\mA_{d-2} &\mB_{d-2}\\
\mB_{d-2}^\tran &\mA_{d-1}
\end{bmatrix}
 \begin{bmatrix}
\vx_{d-2}\\
\vx_{d-1}
\end{bmatrix}\\
&\geq \sum_{k = 1}^{d-2} \begin{bmatrix}
\vx_{k}^\tran&\vx_{k+1}^\tran
\end{bmatrix}
\underbrace{
\begin{bmatrix}
\frac{1}{2}\mA_{k} &\mB_{k}\\
\mB_{k}^\tran &\frac{1}{2}\mA_{k+1}
\end{bmatrix}}_{:=\mM_k}
 \begin{bmatrix}
\vx_{k}\\
\vx_{k+1}
\end{bmatrix}.
\end{align*}
  We will next show that the matrices $\mM_k$  are positive semidefinite which naturally yields the positive semidefiniteness of the matrix $\mM$. 
 
 First, it is not hard to see that $\mA_k$ is invertible and positive definite.  Thus, by Proposition 2.2 in~\cite{gallier2020schur},  a sufficient and necessary condition for   $\mM_k\succeq 0$ is $\frac{1}{2}\mA_{k+1}-\mB_k^\tran\lb\frac{1}{2}\mA_k\rb^{-1}\mB_k\succeq 0$.
The application of~\eqref{eq: recursive relation} yields that
\begin{align*}
    \mL^{k+1} &= L\lb\calX^{k+1}\rb^\tran\lb\mI_{n_{k+1}}\otimes\mL^k\rb L\lb\calX^{k+1}\rb,\\
    \mR^{k+1} &= R\lb \calX^{k+1}\rb\lb\mR^{k+2}\otimes \mI_{n_{k+1}}\rb R\lb \calX^{k+1}\rb^\tran.
\end{align*}
It can be seen that
\begin{align*}
&\mA_{k+1}-4\cdot\mB_k^\tran\mA_k^{-1}\mB_k\\
    &\quad= \mR^{k+2}\otimes\mL^{k+1}- \lb\mR^{k+2}\otimes L\lb\calX^{k+1}\rb^\tran\rb\lb R\lb\calX^{k+1}\rb^\tran{\mR^{k+1}}^{-1}R\lb\calX^{k+1}\rb\otimes\mL^k\rb  \lb\mR^{k+2}\otimes L\lb\calX^{k+1}\rb\rb\\
    &\quad = \lb\mR^{k+2}\otimes L\lb\calX^{k+1}\rb^\tran\rb\lb  \lb {\mR^{k+2}}^{-1}\otimes\mI_{n_{k+1}}- R\lb\calX^{k+1}\rb^\tran{\mR^{k+1}}^{-1}R\lb\calX^{k+1}\rb\rb\otimes\mL^k\rb  \lb\mR^{k+2}\otimes L\lb\calX^{k+1}\rb\rb.
\end{align*}
Let   $\mR^{k+2} = \mU\bm{\Sigma}\mU^\tran$ be the eigenvalue decomposition of $\mR^{k+2}$ and     $R\lb\calX^{k+1}\rb\lb\mU\bm{\Sigma}^{1/2}\otimes\mI_{n_{k+1}}\rb = \mX\bm{\Lambda}\mY^\tran$ be the singular value decomposition of $R\lb\calX^{k+1}\rb\lb\mU\bm{\Sigma}^{1/2}\otimes\mI_{n_{k+1}}\rb$. One has  $\mR^{k+1} = \mX\bm{\Lambda}^2\mX^\tran$ which is  the eigenvalue decomposition of $\mR^{k+1}$.
It follows that
\begin{align*} &R\lb\calX^{k+1}\rb^\tran{\mR^{k+1}}^{-1}R\lb\calX^{k+1}\rb\\
    &~ =
R\lb\calX^{k+1}\rb^\tran\mX\bm{\Lambda}^{-2}\mX^\tran R\lb\calX^{k+1}\rb \\
    & ~ = 
\lb\mU\bm{\Sigma}^{-1/2}\bm{\Sigma}^{1/2}\mU^\tran\otimes\mI_{n_k}\rb R\lb\calX^{k+1}\rb^\tran\mX\bm{\Lambda}^{-2}\mX^\tran R\lb\calX^{k+1}\rb \lb\mU\bm{\Sigma}^{1/2}\bm{\Sigma}^{-1/2}\mU^\tran\otimes\mI_{n_{k+1}}\rb\\
    &~ = \lb\mU\bm{\Sigma}^{-1/2}\otimes \mI_{n_{k+1}}\rb\lb\mY\bm{\Lambda}\mX^\tran \mX\bm{\Lambda}^{-2}\mX^\tran\mX\bm{\Lambda}\mY^\tran \rb\lb\bm{\Sigma}^{-1/2}\mU^\tran\otimes \mI_{n_{k+1}}\rb\\
    &~ = \lb\mU\bm{\Sigma}^{-1/2}\otimes \mI_{n_{k+1}}\rb  \mY\mY^\tran\lb\bm{\Sigma}^{-1/2}\mU^\tran\otimes \mI_{n_{k+1}}\rb.
\end{align*}
Then, one can obtain
\begin{align*}
    &{\mR^{k+2}}^{-1}\otimes\mI_{n_{k+1}}- R\lb\calX^{k+1}\rb^\tran{\mR^{k+1}}^{-1}R\lb\calX^{k+1}\rb\\
    &\quad =
    \mU\bm{\Sigma}^{-1}\mU^\tran\otimes \mI_{n_{k+1}}- \lb\mU\bm{\Sigma}^{-1/2}\otimes \mI_{n_{k+1}}\rb  \mY\mY^\tran\lb\bm{\Sigma}^{-1/2}\mU^\tran\otimes \mI_{n_{k+1}}\rb\\
    &\quad = \lb\mU\bm{\Sigma}^{-1/2}\otimes \mI_{n_{k+1}}\rb \lb \mI_{n_{k+1}r_{k+1}}-\mY\mY^\tran\rb\lb\bm{\Sigma}^{-1/2}\mU^\tran\otimes \mI_{n_{k+1}}\rb\succeq 0.
\end{align*}
Consequently, $\mM_k\succeq 0$ for $k = 1,\cdots,d-2$ which indicates the positive definiteness  of the matrix $\mM$.
\end{proof}

\subsection{Riemannian Metric}
 In this section we verify that the preconditioned metric defined on $T_{\bar{\calX}}\overline{\calM}_{\vr}$, see~\eqref{eq: premetric}, indeed induces a Riemannian metric  on $T_{\lsb\bar{\calX}\rsb}\overline{\calM}_{\vr}/\calG$.
To this end, we first establish the relation
between the horizontal lifts of $\xi_{\lsb\bar{\calX}\rsb}\in T_{\lsb\bar{\calX}\rsb}\overline{\calM}_{\vr}/\calG$ at different elements in $\lsb \bar{\calX}\rsb$.
\begin{lemma}\label{lem: lift equivalent}
Given $\xi_{\lsb\bar{\calX}\rsb}\in T_{\lsb\bar{\calX}\rsb}\overline{\calM}_{\vr}/\calG$,  suppose that  the horizontal lift of $\xi_{\lsb\bar{\calX}\rsb}$  at $\bar{\calX}$ is $\bar{\xi}_{\bar{\calX}}$.  Then for any  $\bar{\calY}  \in \lsb\bar{\calX}\rsb$,  the horizontal lift of $\xi_{\lsb\bar{\calX}\rsb}$ at $\bar{\calY}$ satisfies
\begin{align*}
\bar{\xi}_{\bar{\calY}} = \theta_{A}\lb\bar{\xi}_{\bar{\calX}}\rb := \left\lbrace\xi^1\times_3\mA_1^\tran,\xi^2\times_1\mA_1^{-1}\times_3\mA_2^\tran,\cdots,\xi^d\times_1\mA_{d-1}^{-1}\right\rbrace,
\end{align*}
where $\bar{\xi}_{\bar{\calX}} = \left\lbrace \xi^1,\cdots,\xi^d\right\rbrace$ and $\mA_1,\cdots,\mA_{d-1}$ are invertible matrices such that~\eqref{eq: barX equivalent to barY} holds.
\end{lemma}
\begin{proof}
By the chain rule,
\begin{align*}
    \xi_{\lsb\bar{\calX}\rsb} &= \mathrm{D}\pi\lb\bar{\calX}\rb\lsb\bar{\xi}_{\bar{\calX}}\rsb = \mathrm{D}\pi\lb\theta_A\lb\bar{\calX}\rb\rb \lsb\bar{\xi}_{\bar{\calX}}\rsb\\
   & = \mathrm{D}\pi\lb\bar{\calY}\rb \lsb \mathrm{D}\theta_A\lb\bar{\calX}\rb\lsb\bar{\xi}_{\bar{\calX}}\rsb\rsb\\
  &= \mathrm{D}\pi\lb\bar{\calY}\rb \lsb \theta_{A}\lb\bar{\xi}_{\bar{\calX}}\rb \rsb.
\end{align*}
Thus, under the preconditioned metric,
it remains to show that $\bar{\xi}_{\bar{\calY}}\in\calH_{\bar{\calY}}$. This fact  can be verified as follows.  For $k = 1,\cdots,d-1$,  the left hand side of the equation~\eqref{eq: horizontal equation} can be expressed as 
    \begin{align*}
        &{\calY^{\leq k}}^\tran\lb \mI_{n_k}\otimes \calY^{\leq k-1}\rb L\lb \xi_{\bar{\calX}}^k\times_1\mA_{k-1}^{-1}\times_3\mA_k^\tran\rb{\calY^{\geq k+1}}^\tran\calY^{\geq k+1}\\
       & \quad= \mA_k^\tran L\lb\calX^k\rb^\tran\lb \mI_{n_k}\otimes\mA_{k-1}^{-\tran}\rb\lb\mI_{n_k}\otimes\mA_{k-1}^\tran\mL^{k-1}\mA_{k-1}\rb \lb\mI_{n_k}\otimes\mA_{k-1}^{-1}\rb L\lb\xi_{\bar{\calX}}^k \rb\mA_k\lb \mA_k\mR^{k+1}\mA_{k}^{-\tran}\rb\\
        & \quad = \mA_{k}^\tran {\calX^{\leq k}}^\tran L\lb\xi_{\bar{\calX}}^k \rb{\calX^{\geq k+1}}^\tran\calX^{\geq k+1}\mA_k^{-\tran},
    \end{align*}
    while  the right-hand side of the equation~\eqref{eq: horizontal equation} can be written as
    \begin{align*}
    &{\calY^{\leq k}}^\tran\calY^{\leq k} R\lb \xi_{\bar{\calX}}^{k+1}\times_1\mA_k^{-1}\times_3\mA_{k+1}^\tran \rb \lb {\calY^{\geq k+2}}^\tran\otimes\mI_{n_{k+1}}\rb\calY^{\geq k+1}\\
       & \quad= \mA_k^\tran {\calX^{\leq k}}^\tran\calX^{\geq k} R\lb\xi_{\bar{\calX}}^{k+1} \rb\lb {\calX^{\geq k+2}}^\tran\otimes\mI_{n_{k+1}}\rb\calX^{\geq k+1}\mA_k^{-\tran}.
    \end{align*}
    Since $\mA_k$ is a non-singular matrix, we conclude that
    \begin{align*}
       &{\calY^{\leq k}}^\tran\lb \mI_{n_k}\otimes \calY^{\leq k-1}\rb L\lb \xi_{\bar{\calX}}^k\times_1\mA_{k-1}^{-1}\times_3\mA_k^\tran\rb{\calY^{\geq k+1}}^\tran\calY^{\geq k+1} \\
        &\quad= {\calY^{\leq k}}^\tran\calY^{\leq k} R\lb \xi_{\bar{\calX}}^{k+1}\times_1\mA_k^{-1}\times_3\mA_{k+1}^\tran \rb \lb {\calY^{\geq k+2}}^\tran\otimes\mI_{n_{k+1}}\rb\calY^{\geq k+1}.
    \end{align*}
    As a result, $\bar{\xi}_{\bar{\calY}}\in\calH_{\bar{\calY}}$, which completes the proof. 
\end{proof}
\begin{lemma}
For any $\xi_{\lsb\bar{\calX}\rsb},\eta_{\lsb\bar{\calX}\rsb}\in T_{\lsb\bar{\calX}\rsb}\overline{\calM}_{\vr}/\calG$, define
\begin{align}\label{eq: rie metric}
g_{\lsb\bar{\calX}\rsb}\lb\xi_{\lsb\bar{\calX}\rsb},\eta_{\lsb\bar{\calX}\rsb}\rb :=\bar{g}_{\bar{\calX}}\lb\bar{\xi}_{\bar{\calX}},\bar{\eta}_{\bar{\calX}}\rb,
\end{align}
where $\bar{\xi}_{\bar{\calX}},\bar{\eta}_{\bar{\calX}}\in\calH_{\bar{\calX}}$ are the horizontal lifts of $\xi_{\lsb\bar{\calX}\rsb},\eta_{\lsb\bar{\calX}\rsb}$ at $\bar{\calX}$. Then $g_{\lsb\bar{\calX}\rsb}\lb\cdot,\cdot\rb$ is a Riemannian metric on $T_{\lsb\bar{\calX}\rsb}\overline{\calM}_{\vr}/\calG$.

\end{lemma}
\begin{proof}
With the application of~\cite[Theorem 9.34]{boumal2020introduction},
    we only need to verify the following condition
    \begin{align*}
        \bar{\calX},\bar{\calY}\in \lsb\bar{\calX}\rsb\Rightarrow \bar{g}_{\bar{\calX}}\lb\bar{\xi}_{\bar{\calX}},\bar{\eta}_{\bar{\calX}}\rb = \bar{g}_{\bar{\calY}}\lb\bar{\xi}_{\bar{\calY}},\bar{\eta}_{\bar{\calY}}\rb,
    \end{align*}
    where $\bar{\xi}_{\bar{\calX}} = \left\lbrace \xi^1,\cdots,\xi^d\right\rbrace,\bar{\eta}_{\bar{\calX}} =\left\lbrace \eta^1,\cdots,\eta^d\right\rbrace$ (resp. $\bar{\xi}_{\bar{\calY}},\bar{\eta}_{\bar{\calY}}$) are the horizontal lifts of $\xi_{\lsb\bar{\calX}\rsb},\eta_{\lsb\bar{\calX}\rsb}\in T_{\lsb\bar{\calX}\rsb}\overline{\calM}_{\vr}/\calG$ at $\bar{\calX}$ (resp. $\bar{\calY}$). Given $\bar{\calX},\bar{\calY}\in\lsb\bar{\calX}\rsb$, there exist invertible matrices $\mathrm{A} = \lb\mA_1,\cdots,\mA_{d-1}\rb$ such that~\eqref{eq: barX equivalent to barY} holds. Then one has
\begin{align*}
\bar{g}_{\bar{\calY}}\lb\bar{\xi}_{\bar{\calY}},\bar{\eta}_{\bar{\calY}}\rb 
& = \sum_{k = 1}^d\left\langle \phi\lb\left\lbrace \calY^1,\cdots, \xi^k\times_1\mA_{k-1}^{-1}\times_3\mA_k^\tran,\cdots,\calY^d\right\rbrace\rb, \phi\lb\left\lbrace \calY^1,\cdots, \eta^k\times_1\mA_{k-1}^{-1}\times_3\mA_k^\tran,\cdots,\calY^d\right\rbrace\rb \right\rangle\\
& = \sum_{k = 1}^d\left\langle \phi\lb\left\lbrace \calX^1,\cdots, \xi^k,\cdots,\calX^d\right\rbrace\rb, \phi\lb\left\lbrace \calX^1,\cdots, \eta^k,\cdots,\calX^d\right\rbrace\rb \right\rangle\\
& = \bar{g}_{\bar{\calX}}\lb \bar{\xi}_{\bar{\calX}},\bar{\eta}_{\bar{\calX}}\rb
    \end{align*}
    where the first equation follows from Lemma~\ref{lem: lift equivalent}.
\end{proof}

\subsection{Riemannian Gradient}
Consider a real-valued function $f:\overline{\calM}_{\vr}/\calG\rightarrow\R$ and its lift $\bar{f} = f\circ\pi: \overline{\calM}_{\vr}\rightarrow\R$.
By~\cite[Section 3.6.2]{absil2009optimization},
 the horizontal lift of  the Riemannian  gradient of $f$ can be obtained from the  Riemannian gradient of $\bar{f}$: $$\overline{\grad f\lb \pi\lb\bar{\calX}\rb\rb} = \grad \bar{f}\lb\bar{\calX}\rb.$$  Moreover, the Riemannian gradient of  $\bar{f}$ at $\bar{\calX}$ is  the unique element,  denoted $\grad \bar{f}(\bar{\calX})$,  such that $$\bar{g}_{\bar{\calX}}\lb \bar{\xi},\grad \bar{f}\lb\bar{\calX}\rb\rb  = \mathrm{D}\bar{f}\lb\bar{\calX}\rb\lsb\bar{\xi}\rsb~\mbox{ for all }~ \bar{\xi}\in T_{\bar{\calX}}\overline{\calM}_{\vr}.$$

 For the low rank tensor completion problem in the tensor train format, the function $\bar{f}$ is given by
 \begin{align}\label{eq: function barf}
     \bar{f}\lb\bar{\calX}\rb = \frac{1}{2} \fronorm{\calP_{\Omega}\lb\phi\lb\bar{\calX}\rb\rb- \calP_{\Omega}\lb\calT\rb}^2.
\end{align}
The next lemma gives the expression of $\grad \bar{f}\lb\bar{\calX}\rb$  under the preconditioned metric~\eqref{eq: premetric}.
\begin{lemma}\label{lem: Rie grad derive}
The Riemannian gradient of $\bar{f}$ in~\eqref{eq: function barf} is given by
\begin{align}\label{eq: rg}
\grad \bar{f}\lb\bar{\calX}\rb= \left\lbrace\left\lbrace  L^{-1}\lb\lb\mI_{n_k}\otimes {\mL^{k-1}}^{-1}{\calX^{\leq k-1}}^\tran\rb \lb \calP_{\Omega}\lb\phi\lb\bar{\calX}\rb\rb-\calP_{\Omega}\lb\calT\rb\rb^{<k>}\calX^{\geq k+1}{\mR^{k+1}}^{-1} \rb \right\rbrace_{k = 1}^d\right\rbrace,
\end{align}
where $L^{-1}$ is defined as the inverse operator of the left unfolding operator $L$ such that for any $\calZ\in \R^{r_{k-1}\times n_k\times r_k}$ and $\mZ\in\R^{r_{k-1}n_k\times r_{k}}$,
\begin{align*}
    L^{-1}\lb L\lb\calZ\rb\rb = \calZ~\text{and} ~L\lb L^{-1}\lb\mZ\rb\rb = \mZ.
\end{align*}
\end{lemma}
\begin{proof}
    Let $\bar{\xi} = \left\lbrace\xi^1,\cdots,\xi^d\right\rbrace\in T_{\bar{\calX}}\overline{\calM}_{\vr}$.
We have
\begin{align}\label{eq: present riem grad}
\bar{g}_{\bar{\calX}}\lb \bar{\xi},\grad \bar{f}\lb\bar{\calX}\rb\rb &= \mathrm{D}\bar{f}\lb\bar{\calX}\rb\lsb\bar{\xi}\rsb\notag\\
& = \lim_{t\rightarrow 0}\frac{\bar{f}\lb\bar{\calX}+t\bar{\xi}\rb-\bar{f}\lb\bar{\calX}\rb}{t}\notag\\
& = \sum_{k = 1}^d \left\langle \calP_{\Omega}\lb\phi\lb\bar{\calX}\rb\rb-\calP_{\Omega}\lb\calT\rb,\phi\lb\left\lbrace \calX^1,\cdots,\xi^k,\cdots,\calX^d\right\rbrace\rb \right\rangle\notag\\
&=  \sum_{k = 1}^d \left\langle \lb\calP_{\Omega}\lb\phi\lb\bar{\calX}\rb\rb-\calP_{\Omega}\lb\calT\rb\rb^{<k>}, \lb\mI_{n_k}\otimes\calX^{\leq k-1}\rb L\lb\xi^k\rb{\calX^{\geq k+1}}^\tran\right\rangle.
\end{align}
Moreover, each term in~\eqref{eq: present riem grad} can be expanded as
\begin{small}
\begin{align*}
 \left\langle \lb\mI_{n_k}\otimes \calX^{\leq k-1}{\mL^{k-1}}^{-1}{\calX^{\leq k-1}}^\tran\rb \lb\calP_{\Omega}\lb\phi\lb\bar{\calX}\rb\rb-\calP_{\Omega}\lb\calT\rb\rb^{<k>}\calX^{\geq k+1}{\mR^{k+1}}^{-1} {\calX^{\geq k+1}}^\tran, \lb\mI_{n_k}\otimes\calX^{\leq k-1}\rb L\lb\xi^k\rb{\calX^{\geq k+1}}^\tran\right\rangle,
\end{align*}
\end{small}
which yields the expression of the Riemannian gradient of $\bar{f}$.  
\end{proof}

\subsection{Riemannian Connection and Riemannian Hessian}\label{sec: rieconn and riehess}
For any two vector fields $\xi,\lambda \in T_{\lsb\bar{\calX}\rsb}\overline{\calM}_{\vr}/\calG$, the horizontal lift of the Riemannian connection is given by ~\cite[Proposition 5.3.3]{absil2009optimization}
\begin{align*}
    \overline{\nabla_{\xi_{[\bar{\calX}]}}\lambda} = \calP_{\bar{\calX}}^{\calH}\lb\nabla_{\bar{\xi}_{\bar{\calX}}}\bar{\lambda}  \rb,
\end{align*}
where $\calP_{\bar{\calX}}^{\calH}$ denotes the projection onto the horizontal space, see~\eqref{eq: horizontal projection}. Next, we derive the Riemannian connection $\nabla_{\bar{\xi}_{\bar{\calX}}}\bar{\lambda}  $ on the total space
$\overline{\calM}_{\vr}$ by invoking the Koszul formula. For the preconditioned metric $\bar{g}_{\bar{\calX}}$~\eqref{eq: premetric},  the Koszul formula is 
\begin{align*}
2\bar{g}_{\bar{\calX}}\lb\nabla_{\bar{\xi}_{\bar{\calX}}}\bar{\lambda},\bar{\eta}_{\bar{\calX}}\rb & = \mathrm{D} \bar{g}\lb \bar{\lambda},\bar{\eta}\rb\lb\bar{\calX}\rb\lsb \bar{\xi}_{\bar{\calX}}\rsb + \mathrm{D} \bar{g}\lb \bar{\eta},\bar{\xi}\rb\lb\bar{\calX}\rb\lsb \bar{\lambda}_{\bar{\calX}}\rsb - \mathrm{D} \bar{g}\lb \bar{\xi},\bar{\lambda}\rb\lb\bar{\calX}\rb\lsb \bar{\eta}_{\bar{\calX}}\rsb \\
& \quad-\bar{g}_{\bar{\calX}}\lb\bar{\xi}_{\bar{\calX}}, \lsb \bar{\lambda},\bar{\eta}\rsb_{\bar{\calX}}\rb+\bar{g}_{\bar{\calX}}\lb\bar{\lambda}_{\bar{\calX}}, \lsb \bar{\eta},\bar{\xi}\rsb_{\bar{\calX}}\rb+\bar{g}_{\bar{\calX}}\lb\bar{\eta}_{\bar{\calX}}, \lsb \bar{\xi},\bar{\lambda}\rsb_{\bar{\calX}}\rb,
\end{align*}
where the  definition of the Lie bracket $\lsb\cdot,\cdot\rsb$ can be found for example in~\cite[Section 5.3.1]{absil2009optimization} and $\bar{g}\lb \bar{\lambda},\bar{\eta}\rb\lb\bar{\calX}\rb = \bar{g}_{\bar{\calX}}\lb\bar{\lambda}_{\bar{\calX}},\bar{\eta}_{\bar{\calX}}\rb$.  

A straightforward calculation  shows that
\begin{align*}
\mathrm{D} \bar{g}\lb \bar{\lambda},\bar{\eta}\rb\lb\bar{\calX}\rb\lsb \bar{\xi}_{\bar{\calX}}\rsb & = \bar{g}_{\bar{\calX}}\lb \mathrm{D}\bar{\lambda}\lb\bar{\calX}\rb\lsb \bar{\xi}_{\bar{\calX}}\rsb,  \bar{\eta}_{\bar{\calX}}\rb+\bar{g}_{\bar{\calX}}\lb \bar{\lambda}_{\bar{\calX}},   \mathrm{D}\bar{\eta}\lb\bar{\calX}\rb\lsb \bar{\xi}_{\bar{\calX}}\rsb\rb\\
&\quad +\sum_{k = 1}^d\sum_{j\neq k}\left\langle \phi\lb \left\lbrace\calX^1,\cdots,\xi^j,\cdots,\lambda^k,\cdots,\calX^d\right\rbrace\rb,  \phi\lb \left\lbrace\calX^1,\cdots,\eta^k,\cdots,\calX^d\right\rbrace\rb \right\rangle\\
&\quad + \sum_{k = 1}^d\sum_{j\neq k}\left\langle \phi\lb\left\lbrace\calX^1,\cdots,\lambda^k,\cdots,\calX^d\right\rbrace\rb,  \phi\lb \left\lbrace\calX^1,\cdots,\xi^j,\cdots,\eta^k,\cdots,\calX^d\right\rbrace\rb \right\rangle.
\end{align*}
By definition of the Lie bracket, one can obtain~\cite[Section 5.3.4]{absil2009optimization}
\begin{align*}
    \lsb \bar{\lambda},\bar{\eta}\rsb_{\bar{\calX}} = \mathrm{D}\bar{\eta}\lb\bar{\calX}\rb\lsb\bar{\lambda}_{\bar{\calX}}\rsb - \mathrm{D}\bar{\lambda}\lb\bar{\calX}\rb\lsb\bar{\eta}_{\bar{\calX}}\rsb.
\end{align*}
Consequently, we have
\begin{align*}
2\bar{g}_{\bar{\calX}}\lb\nabla_{\bar{\xi}_{\bar{\calX}}}\bar{\lambda},\bar{\eta}_{\bar{\calX}}\rb & = \mathrm{D} \bar{g}\lb \bar{\lambda},\bar{\eta}\rb\lb\bar{\calX}\rb\lsb \bar{\xi}_{\bar{\calX}}\rsb + \mathrm{D} \bar{g}\lb \bar{\eta},\bar{\xi}\rb\lb\bar{\calX}\rb\lsb \bar{\lambda}_{\bar{\calX}}\rsb - \mathrm{D} \bar{g}\lb \bar{\xi},\bar{\lambda}\rb\lb\bar{\calX}\rb\lsb \bar{\eta}_{\bar{\calX}}\rsb \\
& \quad-\bar{g}_{\bar{\calX}}\lb\bar{\xi}_{\bar{\calX}}, \mathrm{D}\bar{\eta}\lb\bar{\calX}\rb\lsb\bar{\lambda}_{\bar{\calX}}\rsb - \mathrm{D}\bar{\lambda}\lb\bar{\calX}\rb\lsb\bar{\eta}_{\bar{\calX}}\rsb\rb\\
&\quad +\bar{g}_{\bar{\calX}}\lb\bar{\lambda}_{\bar{\calX}}, \mathrm{D}\bar{\xi}\lb\bar{\calX}\rb\lsb\bar{\eta}_{\bar{\calX}}\rsb - \mathrm{D}\bar{\eta}\lb\bar{\calX}\rb\lsb\bar{\xi}_{\bar{\calX}}\rsb\rb\\
&\quad +\bar{g}_{\bar{\calX}}\lb\bar{\eta}_{\bar{\calX}}, \mathrm{D}\bar{\lambda}\lb\bar{\calX}\rb\lsb\bar{\xi}_{\bar{\calX}}\rsb - \mathrm{D}\bar{\xi}\lb\bar{\calX}\rb\lsb\bar{\lambda}_{\bar{\calX}}\rsb\rb\\
& = 2\bar{g}_{\bar{\calX}}\lb\mathrm{D}\bar{\lambda}\lb\bar{\calX}\rb\lsb\bar{\xi}_{\bar{\calX}}\rsb,\bar{\eta}_{\bar{\calX}}\rb 
\\
&\quad +\underbrace{\sum_{k = 1}^d\sum_{j\neq k}\left\langle \phi\lb\left\lbrace\calX^1,\cdots,\xi_{\bar{\calX}}^j,\cdots,\lambda_{\bar{\calX}}^k,\cdots,\calX^d\right\rbrace\rb,  \phi\lb\left\lbrace \calX^1,\cdots,\eta_{\bar{\calX}}^k,\cdots,\calX^d\right\rbrace\rb \right\rangle}_{:= \alpha_1}\\
&\quad + \underbrace{\sum_{k = 1}^d\sum_{j\neq k}\left\langle \phi\lb\left\lbrace\calX^1,\cdots,\lambda_{\bar{\calX}}^k,\cdots,\calX^d\right\rbrace\rb,  \phi\lb \left\lbrace\calX^1,\cdots,\xi_{\bar{\calX}}^j,\cdots,\eta_{\bar{\calX}}^k,\cdots,\calX^d\right\rbrace\rb \right\rangle}_{:=\alpha_2}\\
&\quad +\underbrace{\sum_{k = 1}^d\sum_{j\neq k}\left\langle \phi\lb\left\lbrace\calX^1,\cdots,\lambda_{\bar{\calX}}^j,\cdots,\eta_{\bar{\calX}}^k,\cdots,\calX^d\right\rbrace\rb,  \phi\lb\left\lbrace \calX^1,\cdots,\xi_{\bar{\calX}}^k,\cdots,\calX^d\right\rbrace\rb \right\rangle}_{:=\alpha_3}\\
&\quad + \underbrace{\sum_{k = 1}^d\sum_{j\neq k}\left\langle \phi\lb\left\lbrace\calX^1,\cdots,\eta_{\bar{\calX}}^k,\cdots,\calX^d\right\rbrace\rb,  \phi\lb \left\lbrace\calX^1,\cdots,\lambda_{\bar{\calX}}^j,\cdots,\xi_{\bar{\calX}}^k,\cdots,\calX^d\right\rbrace\rb \right\rangle}_{:=\alpha_4}\\
&\quad -\underbrace{\sum_{k = 1}^d\sum_{j\neq k}\left\langle \phi\lb\left\lbrace\calX^1,\cdots,\eta_{\bar{\calX}}^j,\cdots,\xi_{\bar{\calX}}^k,\cdots,\calX^d\right\rbrace\rb,  \phi\lb\left\lbrace \calX^1,\cdots,\lambda_{\bar{\calX}}^k,\cdots,\calX^d\right\rbrace\rb \right\rangle}_{:=\alpha_5}\\
&\quad - \underbrace{\sum_{k = 1}^d\sum_{j\neq k}\left\langle \phi\lb\left\lbrace\calX^1,\cdots,\xi_{\bar{\calX}}^k,\cdots,\calX^d\right\rbrace\rb,  \phi\lb \left\lbrace\calX^1,\cdots,\eta_{\bar{\calX}}^j,\cdots,\lambda_{\bar{\calX}}^k,\cdots,\calX^d\right\rbrace\rb \right\rangle}_{:=\alpha_6}.
\end{align*}
To obtain a closed-form expression of the Riemannian connection on $\overline{\calM}_{\vr}$ under the preconditioned metric~\eqref{eq: premetric},  we need to rewrite the sum of $\alpha_i$ in the above equation as the form of $\bar{g}_{\bar{\calX}}\lb\bar{\eta}_{\bar{\calX}},\bar{\zeta}_{\bar{\calX}}\rb$ for a specific $\bar{\zeta}_{\bar{\calX}}$. For $\alpha_1$, following the same argument as deriving the Riemannian gradient in Lemma~\ref{lem: Rie grad derive}, one has
\begin{align*}
    \alpha_1 &= \sum_{k = 1}^d\left\langle  \lb\sum_{j\neq k}\phi\lb\left\lbrace\calX^1,\cdots,\xi_{\bar{\calX}}^j,\cdots,\lambda_{\bar{\calX}}^k,\cdots,\calX^d\right\rbrace\rb\rb^{<k>}, \lb\mI_{n_k}\otimes\calX^{\leq k-1}\rb L\lb\eta_{\bar{\calX}}^k\rb{\calX^{\geq k+1}}^\tran\right\rangle\\
    & = \bar{g}_{\bar{\calX}}\lb \bar{\eta}_{\bar{\calX}},  \bar{\gamma}_{\bar{\calX}} \rb,
\end{align*}
where the left unfolding of the $k$-th element in $\bar{\gamma}_{\bar{\calX}}$ is
\begin{align*}
    L\lb\gamma_{\bar{\calX}}^k\rb = \lb\mI_{n_k}\otimes {\mL^{k-1}}^{-1}{\calX^{\leq k-1}}^\tran\rb \lb\sum_{j\neq k}\phi\lb\left\lbrace\calX^1,\cdots,\xi_{\bar{\calX}}^j,\cdots,\lambda_{\bar{\calX}}^k,\cdots,\calX^d\right\rbrace\rb\rb^{<k>} \calX^{\geq k+1}{\mR^{k+1}}^{-1}. 
\end{align*}
For $\alpha_2$, it can be expressed as
\begin{align*}
    \alpha_2  
    & = \sum_{k = 1}^d\sum_{j < k}\left\langle \lb\mI_{n_k}\otimes \calX^{\leq k-1}\rb L\lb \lambda_{\bar{\calX}}^k\rb{\calX^{\geq k+1}}^\tran,  \lb\mI_{n_k}\otimes\calX_{\xi}^{\leq k-1,j}\rb L\lb \eta_{\bar{\calX}}^k\rb{\calX^{\geq k+1} }^\tran\right\rangle\\
    & \quad+ \sum_{k = 1}^d\sum_{j > k}\left\langle \lb\mI_{n_k}\otimes \calX^{\leq k-1}\rb L\lb \lambda_{\bar{\calX}}^k\rb{\calX^{\geq k+1}}^\tran,  \lb\mI_{n_k}\otimes\calX^{\leq k-1}\rb L\lb \eta_{\bar{\calX}}^k\rb{\calX_{\xi}^{\geq k+1,j} }^\tran\right\rangle\\
    & =  \sum_{k = 1}^d\sum_{j < k}\left\langle \lb\mI_{n_k}\otimes \calX^{\leq k-1}{\mL^{k-1}}^{-1} {\calX_{\xi}^{\leq k-1,j}}^\tran\calX^{\leq k-1}\rb L\lb \lambda_{\bar{\calX}}^k\rb{\calX^{\geq k+1}}^\tran,  \lb\mI_{n_k}\otimes\calX^{\leq k-1}\rb L\lb \eta_{\bar{\calX}}^k\rb{\calX^{\geq k+1} }^\tran\right\rangle\\
    & \quad+ \sum_{k = 1}^d\sum_{j > k}\left\langle \lb\mI_{n_k}\otimes \calX^{\leq k-1}\rb L\lb \lambda_{\bar{\calX}}^k\rb{\calX^{\geq k+1}}^\tran{\calX_{\xi}^{\geq k+1,j} } {\mR^{k+1}}^{-1}{\calX^{\geq k+1}}^\tran,  \lb\mI_{n_k}\otimes\calX^{\leq k-1}\rb L\lb \eta_{\bar{\calX}}^k\rb{\calX^{\geq k+1}}^\tran\right\rangle,
\end{align*}
where the modified interface matrices  are defined as
\begin{align*}
\calX_{\xi}^{\leq k,j}\in\R^{n_1n_2\cdots n_k\times r_k}: \calX_{\xi}^{\leq k,j}\lb i_1,i_2,\cdots,i_k;:\rb &= \calX^1\lb i_1\rb\cdots\xi_{\bar{\calX}}^j\lb i_j\rb\cdots\calX^k\lb i_k\rb,\\
    \calX_{\xi}^{\geq k,j}\in\R^{n_kn_{k+1}\cdots n_d\times r_{k-1}}:\calX_{\xi}^{\geq k,j}\lb i_k,i_{k+1},\cdots,i_d;:\rb &= \lsb\calX^k\lb i_k\rb\cdots\xi_{\bar{\calX}}^{j}\lb i_{j}\rb\cdots\calX^d\lb i_d\rb\rsb^\tran.
\end{align*}
Similarly, one can
obtain $\alpha_2 = \bar{g}_{\bar{\calX}}\lb\bar{\eta}_{\bar{\calX}},\bar{\upsilon}_{\bar{\calX}}\rb$, 
where the left unfolding of the $k$-th tensor in $\bar{\upsilon}_{\bar{\calX}}$ is
\begin{align*}
    L\lb \upsilon_{\bar{\calX}}^k \rb = \sum_{j<k}\lb\mI_{n_k}\otimes {\mL^{k-1}}^{-1} {\calX_{\xi}^{\leq k-1,j}}^\tran\calX^{\leq k-1}\rb L\lb \lambda_{\bar{\calX}}^k\rb +\sum_{j>k} L\lb \lambda_{\bar{\calX}}^k\rb{\calX^{\geq k+1}}^\tran{\calX_{\xi}^{\geq k+1,j} } {\mR^{k+1}}^{-1}.
\end{align*}
The remaining four terms $\alpha_3,\cdots,\alpha_6$ can be rewritten in the same way and the details are omitted.  As a result,   we get
\begin{align*}
2\bar{g}_{\bar{\calX}}\lb\nabla_{\bar{\xi}_{\bar{\calX}}}\bar{\lambda},\bar{\eta}_{\bar{\calX}}\rb &= 2\bar{g}_{\bar{\calX}}\lb\bar{\eta}_{\bar{\calX}}, \mathrm{D}\bar{\lambda}\lb\bar{\calX}\rb\lsb\bar{\xi}_{\bar{\calX}}\rsb\rb+\bar{g}_{\bar{\calX}}\lb\bar{\eta}_{\bar{\calX}},\bar{\zeta}_{\bar{\calX}}\rb,
\end{align*}
where the left unfolding of the $k$-th element of  $\bar{\zeta}_{\bar{\calX}}$ is
\begin{align*}
    L\lb \zeta_{\bar{\calX}}^k\rb &= \lb\mI_{n_k}\otimes {\mL^{k-1}}^{-1}{\calX^{\leq k-1}}^\tran\rb \lb\sum_{j\neq k}\phi\lb\left\lbrace\calX^1,\cdots,\xi_{\bar{\calX}}^j,\cdots,\lambda_{\bar{\calX}}^k,\cdots,\calX^d\right\rbrace\rb\rb^{<k>} \calX^{\geq k+1}{\mR^{k+1}}^{-1}\\
    &\quad +\lb\mI_{n_k}\otimes {\mL^{k-1}}^{-1}{\calX^{\leq k-1}}^\tran\rb \lb\sum_{j\neq k}\phi\lb\left\lbrace\calX^1,\cdots,\lambda_{\bar{\calX}}^j,\cdots,\xi_{\bar{\calX}}^k,\cdots,\calX^d\right\rbrace\rb\rb^{<k>} \calX^{\geq k+1}{\mR^{k+1}}^{-1}\\
    &\quad + \sum_{j<k}\lb\mI_{n_k}\otimes {\mL^{k-1}}^{-1} {\calX_{\xi}^{\leq k-1,j}}^\tran\calX^{\leq k-1}\rb L\lb \lambda_{\bar{\calX}}^k\rb +\sum_{j>k} L\lb \lambda_{\bar{\calX}}^k\rb{\calX^{\geq k+1}}^\tran{\calX_{\xi}^{\geq k+1,j} } {\mR^{k+1}}^{-1}\\
    &\quad +\sum_{j<k}\lb\mI_{n_k}\otimes {\mL^{k-1}}^{-1} {\calX_{\lambda}^{\leq k-1,j}}^\tran\calX^{\leq k-1}\rb L\lb \xi_{\bar{\calX}}^k\rb +\sum_{j>k} L\lb \xi_{\bar{\calX}}^k\rb{\calX^{\geq k+1}}^\tran{\calX_{\lambda}^{\geq k+1,j} } {\mR^{k+1}}^{-1}\\
    &\quad -\sum_{j<k} \lb\mI_{n_k}\otimes {\mL^{k-1}}^{-1}{\calX_{\lambda}^{\leq k-1,j}}^\tran\calX_{\xi}^{\leq k-1,j}\rb L\lb\calX^k\rb-\sum_{j>k}L\lb\calX^k\rb{\calX^{\geq k+1,j}_{\lambda}}^\tran\calX^{\geq k+1,j}_{\xi}{\mR^{k+1}}^{-1}\\
    &\quad -\sum_{j<k} \lb \mI_{n_k}\otimes{\mL^{k-1}}^{-1}{\calX_{\xi}^{\leq k-1,j}}^\tran\calX_{\lambda}^{\leq k-1,j}\rb L\lb\calX^k\rb-\sum_{j>k}L\lb\calX^k\rb{\calX^{\geq k+1,j}_{\xi}}^\tran\calX^{\geq k+1,j}_{\lambda}{\mR^{k+1}}^{-1}.
\end{align*}

The above argument leads to the following lemma about the Riemannian connection on $\overline{\calM}_{\vr}$.
\begin{lemma}\label{lem: riemannian connection}
    The Riemannian connection on $\overline{\calM}_{\vr}$ under the preconditioned metric~\eqref{eq: premetric}  is given by
\begin{align}\label{eq: riemannian connection}
\nabla_{\bar{\xi}_{\bar{\calX}}}\bar{\lambda} = \mathrm{D}\bar{\lambda}\lb\bar{\calX}\rb\lsb\bar{\xi}_{\bar{\calX}}\rsb+\frac{1}{2} \bar{\zeta}_{\bar{\calX}}.
\end{align}
\end{lemma}
The  horizontal lift  of the Riemannian Hessian of  $f$ under the  metric $\bar{g}$~\eqref{eq: premetric} is given by~\cite{absil2009optimization}
\begin{align*}
    \overline{\Hess f\lb \lsb\bar{\calX}\rsb\rb\lsb\xi_{\lsb\bar{\calX}\rsb}\rsb} = \calP_{\bar{\calX}}^{\calH}\lb \Hess \bar{f}\lb \bar{\calX}\rb\lsb \bar{\xi}_{\bar{\calX}}\rsb\rb,
\end{align*}
where $\bar{\xi}_{\bar{\calX}}$ is the horizontal lift of $\xi_{\lsb\bar{\calX}\rsb}$ at $\bar{\calX}$.
According to~\cite[Definition 5.5.1]{absil2009optimization},  the Riemannian
Hessian of $\bar{f}$ at $\bar{\calX}$ can be computed via the Riemannian connection~\eqref{eq: riemannian connection}:
\begin{align*}
    \Hess \bar{f}\lb \bar{\calX}\rb\lsb \bar{\xi}_{\bar{\calX}}\rsb = \nabla_{\bar{\xi}_{\bar{\calX}}}\grad \bar{f}.
\end{align*}

\subsection{Retraction and Vector Transport}
 Let $\calM$ be a general manifold. The tangent space of the manifold $\calM$ at $x$ is denoted by $T_x\calM$ and the tangent bundle is denoted by $T\calM$. 
A retraction~\cite{absil2009optimization} is a smooth mapping from the tangent bundle to the manifold such that, for all $x\in\calM$, $\eta_x\in T_x\calM$, (i) $R_{x}\lb 0_{x}\rb = x$ where $0_{x}$ denotes the the zero element of $T_{x}\calM$, and (ii) $\frac{d}{dt}R_{x}\lb t\eta_{x}\rb\vert_{t = 0} = \eta_{x}$.
By~\cite[Section 4.1.2]{absil2009optimization}, a retraction on the quotient manifold $\overline{\calM}_{\vr}/\calG$ is given by
\begin{align}\label{eq: retraction quotient}
R^{Q}_{\lsb\bar{\calX}\rsb}\lb\xi_{\lsb\bar{\calX}\rsb}\rb = \pi\lb R^{Q}_{\bar{\calX}}\lb\bar{\xi}_{\bar{\calX}}\rb\rb  =\lsb \bar{R}^{Q}_{\bar{\calX}}\lb\bar{\xi}_{\bar{\calX}}\rb\rsb,  
\end{align}
where $\bar{R}^Q$ is a retraction   defined on the total space $\overline{\calM}_{\vr}$,
\begin{align}\label{eq: retraction total space}
    \bar{R}^{Q}_{\bar{\calX}}\lb\bar{\xi}_{\bar{\calX}}\rb =\left\lbrace \left\lbrace \calX^k+\xi^k\right\rbrace_{k= 1}^d\right\rbrace.
\end{align}

A vector transport~\cite{absil2009optimization}: $T\calM\oplus T\calM\rightarrow T\calM: \lb\eta_x,\xi_x\rb\rightarrow\calT_{\eta_x}\xi_x$ associated with a retraction $R$ is a smooth mapping 
such that, for all $x\in \calM$, $\eta_x,\xi_x\in T_x\calM$, (i) $\calT_{\eta_x}\xi_x\in T_{R_x\lb\eta_x\rb}\calM$, (ii) $\calT_{0_x}\xi_x = \xi_x$, and (iii) $\calT_{\eta_x}$ is a linear map. Here, $\oplus$ denotes the Whitney sum~\cite[P.169]{absil2009optimization}.
 As shown in~\cite[Section 8.1.4]{absil2009optimization}, 
\begin{align*}
\overline{\calT^Q_{\eta_{\lsb\bar{\calX}\rsb}}\xi_{\lsb\bar{\calX}\rsb}}_{\bar{\calX}+\bar{\eta}_{\bar{\calX}}}:= \calP^\calH_{\bar{\calX}+\bar{\eta}_{\bar{\calX}}}\bar{\xi}_{\bar{\calX}} 
\end{align*}
defines a vector transport on $T_{\lsb\bar{\calX}\rsb}\overline{\calM}_{\vr}/\calG$. In fact, it can be shown that the above vector transport coincides with  the differential of the retraction~\eqref{eq: retraction quotient}:
\begin{align*}
    \mathrm{D} R^{Q}_{\lsb\bar{\calX}\rsb}\lb \eta_{\lsb\bar{\calX}\rsb}\rb\lsb \xi_{\lsb\bar{\calX}\rsb}\rsb& = \frac{d}{d t}R^{Q}_{\lsb\bar{\calX}\rsb}\lb \eta_{\lsb\bar{\calX}\rsb}+t\xi_{\lsb\bar{\calX}\rsb}\rb\bigg|_{t = 0}\\
    &= \frac{d}{dt}\pi\circ\bar{R}^{Q}_{\bar{\calX}}\lb \bar{\eta}_{\bar{\calX}}+t\bar{\xi}_{\bar{\calX}}\rb\bigg|_{t = 0}\\
    & = \mathrm{D}\pi\circ\bar{R}^{Q}_{\bar{\calX}}\lb \bar{\eta}_{\bar{\calX}}\rb\lsb\bar{\xi}_{\bar{\calX}}\rsb\\
    & = \mathrm{D}\pi\lb\bar{R}^{Q}_{\bar{\calX}}\lb \bar{\eta}_{\bar{\calX}}\rb \rb\lsb\mathrm{D}\bar{R}^{Q}_{\bar{\calX}}\lb \bar{\eta}_{\bar{\calX}}\rb\lsb\bar{\xi}_{\bar{\calX}}\rsb\rsb\\
    & = \mathrm{D}\pi\lb \bar{\calX}+\bar{\eta}_{\bar{\calX}}\rb\lsb\bar{\xi}_{\bar{\calX}}\rsb\\
    & = \mathrm{D}\pi\lb \bar{\calX}+\bar{\eta}_{\bar{\calX}}\rb\lsb\calP^\calH_{\bar{\calX}+\bar{\eta}_{\bar{\calX}}}\bar{\xi}_{\bar{\calX}} \rsb.
\end{align*}

\section{Algorithms}

\subsection{Riemannian Gradient Descent and Conjugate Gradient Methods}
\paragraph{Riemannian gradient descent method.}
 The Riemannian gradient descent method under the quotient geometry (RGD (Q)) for the low-rank tensor completion problem  is presented in Algorithm~\ref{alg: RGD(Q)}. Let $\bar{\calX}_t$ be the current estimator.  RGD (Q) updates $\bar{\calX}_t$ along the negative Riemannian gradient direction given in~\eqref{eq: rg}, followed by retraction.
The step size $\alpha_t$ at $t$-th iteration is calculated by the standard backtracking line search procedure.  For $t = 0$, we take $\alpha_0^{0} = 1$ as the initial step size. For $t\geq 1$, the following  Riemannian Barzilai–Borwein (RBB) step size rule~\cite{iannazzo2018riemannian} without safeguard will be considered, 
\begin{align*}
\alpha_t^0 := \frac{\bar{g}_{\bar{\calX}_t}\lb s_t,s_t\rb}{\left\vert\bar{g}_{\bar{\calX}_t}\lb s_t,y_t\rb\right\vert},
\end{align*} 
where $s_t = \alpha_{t-1}\calP_{\bar{\calX}_t}^{\calH}\lb\bar{\xi}_{t-1}\rb$ and $y_t = \bar{\xi}_t-\calP_{\bar{\calX}_t}^{\calH}\lb\bar{\xi}_{t-1}\rb$.
{\LinesNumberedHidden
\begin{algorithm}[ht!]
	\caption{Riemannian Gradient Descent (RGD (Q))}
	\label{alg: RGD(Q)}
	\KwIn{Initial point $\bar{\calX}_0\in \overline{\calM}_{\vr}$, $\beta,\sigma\in (0,1)$, tolerance $\varepsilon>0$}
	\For{$t = 0,1,\cdots$}{ 
	Compute $\bar{\xi}_t = -\grad\bar{f}\lb\bar{\calX}_t\rb$ using~\eqref{eq: rg};\\
    Check convergence: if $\sqrt{\bar{g}_{\bar{\calX}_t}\lb\bar{\xi}_t,\bar{\xi}_t \rb}<\varepsilon$, then break;\\
    Backtracking line search: given $\alpha_t^0$, find the smallest $\ell\geq 0$ such that for $\alpha_t = \alpha_t^0\beta^{\ell}$, 
    \begin{align*}
        \bar{f}\lb\bar{\calX}_t\rb-\bar{f}\lb \bar{R}_{ \bar{\calX}_t}^Q\lb \alpha_t\bar{\xi}_t \rb\rb  \geq \sigma\alpha_t\bar{g}_{\bar{\calX}_t}\lb \bar{\xi}_t,\bar{\xi}_t\rb.
    \end{align*}
    \\
    Update: $\bar{\calX}_{t+1} = \bar{R}_{ \bar{\calX}_t}^Q\lb \alpha_t\bar{\xi}_t \rb$;
	} 
 \KwOut{$\bar{\calX}_t\in\overline{\calM}_{\vr}$.}
\end{algorithm}}


 Let $n := \max_k n_k$ and $r := \max_k r_k$. Calculating the Riemannian gradient~\eqref{eq: rg} can be split into three steps. We first compute the term $\lb\mI_{n_k}\otimes {\calX_t^{\leq k-1}}^\tran\rb \lb  \calP_{\Omega}\lb\phi\lb\bar{\calX}_t\rb\rb-\calP_{\Omega}\lb\calT\rb \rb^{<k>}\calX_t^{\geq k+1}$ which can be done efficiently by exploiting the sparsity of $\calP_{\Omega}\lb\phi\lb\bar{\calX}_t\rb\rb-\calP_{\Omega}\lb\calT\rb$. As shown in~\cite[Algorithm 2]{steinlechner2016riemannian}, the computation of this term requires $O\lb d\vert\Omega\vert r^2\rb$ floating point operations (flops). Then we form the matrices $\mL_t^k$ and $\mR_t^k$ recursively via the equation~\eqref{eq: recursive relation} which costs $O\lb dnr^3 \rb$ flops, see Algorithm~\ref{alg: calculate LR}. Finally, computing the inverse of $\mL_t^k, \mR^k_t$ and the matrix product  requires $O\lb dnr^3+dr^3 \rb$ flops. Hence the total cost of computing the Riemannian gradient is $O\lb d\vert\Omega\vert r^2+dnr^3 \rb$.
The main computational complexity of calculating the backtracking line search procedure lies in computing the horizontal projection in the RBB step size which  requires $O\lb dnr^6\rb$ flops to form the coefficient matrix in~\eqref{eq: linear equation} and $O\lb dr^6\rb$ flops to solve the linear system~\cite[P.96]{quarteroni2010numerical}.  In conclusion, the total cost of one iteration of RGD (Q) is $O\lb d\vert\Omega\vert r^2+dnr^6 \rb$. Note that the  information theoretic minimum of the number of observations $|\Omega|$ should be $O(dnr^2)$ (the dimension of the manifold~\eqref{eq: embedded manifold}), leading to the overall $O(d^2nr^4+dnr^6)$ computation complexity. In addition, if we solve the block diagonal linear system~\eqref{eq: linear equation} by the conjugate gradient method (note that the coefficient matrix is positive definite as shown in Lemma~\ref{lem: positive definite}), the overall computational complexity of RGD (Q) will be $O(d^2nr^4)$, which is  the same as that of Riemannian conjugate gradient algorithm presented in~\cite{steinlechner2016riemannian}.  
{\LinesNumberedHidden
\begin{algorithm}[ht!]
	\label{alg: calculate LR}
	\caption{Computation of the interface matrix products}
	\KwIn{ $\bar{\calX} = \left\lbrace\calX^1,\cdots,\calX^d\right\rbrace$ }
  $\mL^0 = \mR^{d+1} = 1$, $\mL^1 = L\lb \calX^1\rb^\tran L\lb \calX^1\rb$, $\mR^d = R\lb\calX^d\rb R\lb\calX^d\rb^\tran$; \\
	\For{$k = 2,\cdots,d-1$}{  
 $\mL^k = L\lb \calX^k\rb^\tran\lb \mI_{n_k}\otimes\mL^{k-1}\rb L\lb \calX^k\rb$, $\mR^{d-k+1} = R\lb\calX^{d-k+1}\rb\lb \mR^{d-k+2}\otimes\mI_{n_{d-k+1}}\rb R\lb\calX^{d-k+1}\rb^\tran$;
	}
 \KwOut{$\left\lbrace \mL^k\right\rbrace_{k = 0}^{d-1},\left\lbrace\mR^k\right\rbrace_{k = 2}^{d+1}$.}
\end{algorithm}}


\paragraph{Riemannian conjugate gradient method.}
The Riemannian conjugate gradient method  under the quotient geometry (RCG (Q)) for low-rank tensor completion in the tensor train format is presented in Algorithm~\ref{alg: RCG(Q)}. Let $\bar{\calX}_t$ be the current estimate.  
The conjugate direction at  $t$-th iteration is 
\begin{align*}
\bar{\eta_{t}} =
-\grad\bar{f}\lb\bar{\calX}_{t}\rb+\beta_{t}\calP_{\bar{\calX}_{t}}^\calH\lb \bar{\eta}_{t-1}\rb.
\end{align*}
 In this paper,  the following  modified Hestenes-Stiefel rule~\cite{gilbert1992global} will be chosen  to calculate $\beta_t$:
\begin{align}\label{eq: HS rule}
\beta_t = \max\left\lbrace  0,\frac{\bar{g}_{\bar{\calX}_t}\lb \grad\bar{f}\lb\bar{\calX}_{t}\rb-\calP_{\bar{\calX}_{t}}^\calH\lb\grad\bar{f}\lb\bar{\calX}_{t-1}\rb\rb ,  \grad\bar{f}\lb\bar{\calX}_{t}\rb\rb}{\bar{g}_{\bar{\calX}_t}\lb \grad\bar{f}\lb\bar{\calX}_{t}\rb-\calP_{\bar{\calX}_{t}}^\calH\lb\grad\bar{f}\lb\bar{\calX}_{t-1})\rb\rb, \calP_{\bar{\calX}_{t}}^\calH\lb \bar{\eta}_{t-1}\rb \rb} \right\rbrace.
\end{align}
Then, RCG (Q) updates $\bar{\calX}_t$ along the conjugate direction, followed by retraction. 
Regarding the step size, the optimal choice of $\alpha_t$ would be the minimizer of the objective function: $\alpha_t = \arg\min_{\alpha\in \R}\bar{f}\lb \bar{R}^Q_{\bar{\calX_t}}\lb\alpha\bar{\eta}_t\rb\rb$. 
However, for the retraction~\eqref{eq: retraction total space}, the exact  $\alpha_t$ is expensive to calculate, since the minimization problem is a degree $d^2$ polynomial in $\alpha$. Inspired by~\cite{kasai2016low}, we instead consider a degree $2$ polynomial  approximation  of the  minimization problem,
\begin{align*}
\arg\min_{\alpha\in\R}~\fronorm{\calP_{\Omega}\lb\phi\lb \bar{\calX}_t\rb +\alpha\sum_{k = 1}^d\phi\lb\left\lbrace\calX_t^1,\cdots,\eta_t^k,\cdots,\calX_t^d\right\rbrace\rb\rb-\calP_{\Omega}\lb \calT\rb}^2,
\end{align*}
which admits a closed-form solution given by
\begin{align}\label{eq: alphat}
\alpha_t = \frac{\left\langle \calP_{\Omega}\lb \sum_{k = 1}^d\phi\lb\left\lbrace\calX_t^1,\cdots,\eta_t^k,\cdots,\calX_t^d\right\rbrace\rb \rb,  \calP_{\Omega}\lb \calT\rb-\calP_{\Omega}\lb\phi\lb\bar{\calX}_t\rb\rb  \right\rangle}{\left\langle \calP_{\Omega}\lb \sum_{k = 1}^d\phi\lb\left\lbrace\calX_t^1,\cdots,\eta_t^k,\cdots,\calX_t^d\right\rbrace\rb \rb,  \calP_{\Omega}\lb \sum_{k = 1}^d\phi\lb\left\lbrace\calX_t^1,\cdots,\eta_t^k,\cdots,\calX_t^d\right\rbrace\rb \rb \right\rangle}.
\end{align}
{\LinesNumberedHidden
\begin{algorithm}[ht!]
	\caption{Riemannian Conjugate Gradient (RCG (Q))}
	\label{alg: RCG(Q)}
	\KwIn{Initial point $\bar{\calX}_0 \in\overline{\calM}_{\vr}$, tolerance $\varepsilon>0$}
    Compute $\bar{\eta}_{-1} = 0$;\\
	\For{$t = 0,1,\cdots$}{ 
        Set $\bar{\xi}_{t} = -\grad\bar{f}\lb\bar{\calX}_{t}\rb$ using~\eqref{eq: rg};\\
         Check convergence: if $\sqrt{\bar{g}_{\bar{\calX}_{t}}\lb\bar{\xi}_{t},\bar{\xi}_{t} \rb}<\varepsilon$, then break;\\
        Compute $\beta_{t}$ using~\eqref{eq: HS rule} and set
$\bar{\eta}_{t} = -\bar{\xi}_t+\beta_{t}\calP_{\bar{\calX}_{t}}^{\calH}\lb \bar{\eta}_{t-1}\rb$;\\
	Compute step size $\alpha_t$ using~\eqref{eq: alphat};\\
        Update: $\bar{\calX}_{t+1} = \bar{R}_{ \bar{\calX}_t}^Q\lb \alpha_t\bar{\eta}_t \rb$;\\
        }
\KwOut{$\bar{\calX}_{t}\in\overline{\calM}_{\vr}$.}
\end{algorithm}}

The main computational complexity of computing the conjugate direction lies in the calculation of the Riemannian gradient  and  the horizontal projection which requires $O\lb d\vert\Omega\vert r^2+ dnr^6+dr^6\rb$ flops. 
Additionally, it takes $O\lb d\vert\Omega\vert r^2\rb$ flops to compute the step size $\alpha_t$~\cite[Section 4.5]{steinlechner2016riemannian}. Thus,  the leading order per iteration computational cost of RCG (Q) is $ O \lb d\vert\Omega\vert r^2+dnr^6 \rb$.

\subsection{Riemannian Gauss-Newton Method}\label{sec: RGNQ}
The geometric Newton method for a real-valued function $f:\overline{\calM}_{\vr}/\calG\rightarrow\R$ requires to compute the Newton direction $\xi_{\lsb \bar{\calX}\rsb}\in T_{\lsb\bar{\calX}\rsb}\overline{\calM}_{\vr}/\calG$ which is  the solution of the  equation
\begin{align*}
\Hess f\lb\lsb\bar{\calX}\rsb\rb\lsb\xi_{\lsb \bar{\calX}\rsb}\rsb = -\grad f\lb\lsb\bar{\calX}\rsb\rb.
\end{align*}
Lifting both sides of this equation to the horizontal space at $\bar{\calX}$ yields the  linear equation~\cite[Section 9.12]{boumal2020introduction}
\begin{align}\label{eq: newton equa}
\calP_{\bar{\calX}}^{\calH}\lb \Hess \bar{f}\lb\bar{\calX}\rb\lsb\bar{\xi}_{\bar{\calX}}\rsb\rb = -\grad \bar{f}\lb\bar{\calX}\rb,
\end{align}
where $\bar{f} = f\circ\pi:\overline{\calM}_{\vr}\rightarrow\R$ and  $\bar{\xi}_{\bar{\calX}}$ is the horizontal lift of $\xi_{\lsb \bar{\calX}\rsb}$ at $\bar{\calX}$. As discussed in Section~\ref{sec: rieconn and riehess}, calculating  the Riemannian Hessian is quite involved. Instead, we consider  the Riemannian Gauss-Newton method which is an approximation of the geometric Newton method for the case when $\bar{f}\lb\bar{\calX}\rb = \frac{1}{2}\fronorm{\bar{F}\lb\bar{\calX}\rb}^2$.

Notice that the equation~\eqref{eq: newton equa} is equivalent to finding a $\bar{\xi}_{\bar{\calX}}\in\calH_{\bar{\calX}}$ such that for all $\bar{\eta}\in T_{\bar{\calX}}\overline{\calM}_{\vr}$,
\begin{align*}
    0 &= \bar{g}_{\bar{\calX}}\lb \calP_{\bar{\calX}}^{\calH}\lb \Hess \bar{f}\lb\bar{\calX}\rb\lsb\bar{\xi}_{\bar{\calX}}\rsb\rb , \bar{\eta} \rb +\bar{g}_{\bar{\calX}}\lb\grad \bar{f}\lb\bar{\calX}\rb,\bar{\eta}\rb\\
    &=\bar{g}_{\bar{\calX}}\lb \calP_{\bar{\calX}}^{\calH}\lb \Hess \bar{f}\lb\bar{\calX}\rb\lsb\bar{\xi}_{\bar{\calX}}\rsb\rb , \calP_{\bar{\calX}}^{\calH}\bar{\eta} \rb +\bar{g}_{\bar{\calX}}\lb\grad \bar{f}\lb\bar{\calX}\rb,\calP_{\bar{\calX}}^{\calH}\bar{\eta}\rb\\
& = \bar{g}_{\bar{\calX}}\lb  \Hess \bar{f}\lb\bar{\calX}\rb\lsb\bar{\xi}_{\bar{\calX}}\rsb ,\calP_{\bar{\calX}}^{\calH} \bar{\eta} \rb +\bar{g}_{\bar{\calX}}\lb\grad \bar{f}\lb\bar{\calX}\rb,\calP_{\bar{\calX}}^{\calH}\bar{\eta}\rb,
\end{align*}
or equivalently, 
\begin{align*}
    0  = \mathrm{D} \bar{f}\lb\bar{\calX}\rb\lsb\calP_{\bar{\calX}}^{\calH}\bar{\eta}\rsb+ \nabla^2\bar{f}\lb\bar{\calX}\rb\lsb \bar{\xi}_{\bar{\calX}},\calP_{\bar{\calX}}^{\calH} \bar{\eta}\rsb, 
\end{align*}
where  the  definition of the second covariant derivative $\nabla^2\bar{f}\lb\bar{\calX}\rb\lsb \cdot,\cdot\rsb$ can be found for example in~\cite[Section 5.6]{absil2009optimization}.
 Approximating $ \nabla^2\bar{f}\lb\bar{\calX}\rb\lsb \bar{\xi}_{\bar{\calX}},\calP_{\bar{\calX}}^{\calH} \bar{\eta}\rsb$ by $\left\langle \mathrm{D} \bar{F}\lb\bar{\calX}\rb\lsb\bar{\xi}_{\bar{\calX}}\rsb,  \mathrm{D} \bar{F}\lb\bar{\calX}\rb\lsb\calP_{\bar{\calX}}^{\calH} \bar{\eta} \rsb \right\rangle$ yields the Gauss-Newton equation~\cite[Section 8.4.1]{absil2009optimization}
\begin{align}\label{eq: GNQ equation}
0 &= \left\langle   \mathrm{D} \bar{F}\lb\bar{\calX}\rb\lsb\bar{\eta}\rsb,\bar{F}\lb\bar{\calX}\rb \right\rangle + \left\langle   \mathrm{D} \bar{F}\lb\bar{\calX}\rb\lsb \bar{\eta} \rsb, \mathrm{D} \bar{F}\lb\bar{\calX}\rb\lsb\bar{\xi}_{\bar{\calX}}\rsb \right\rangle, ~\text{for all}~ \bar{\eta}\in \calH_{\bar{\calX}}.
\end{align}
The Riemannian Gauss-Newton method under the quotient geometry (RGN (Q)) is presented in Algorithm~\ref{alg: RGN(Q)}. For the low-rank tensor train tensor completion problem, the efficient solution of~\eqref{eq: GNQ equation} will be presented in Section~\ref{sec: implemention detail} after we  establish the equivalence of the Riemannian Gauss-Newton methods under the quotient and embedded geometries.
{\LinesNumberedHidden
\begin{algorithm}[ht!]
	\label{alg: RGN(Q)}
	\caption{Riemannian Gauss-Newton under the quotient geometry (RGN (Q))}
	\KwIn{Initial point $\bar{\calX}_0 \in\overline{\calM}_{\vr}$, tolerance $\varepsilon>0$}
	\For{$t = 0,1,\cdots$}{ 
	Solving the Gauss Newton equation~\eqref{eq: GNQ equation} gives $\bar{\xi}_t\in \calH_{\bar{\calX}_t}$.\\
 Check convergence: if $\sqrt{\bar{g}_{\bar{\calX}_{t}}\lb\bar{\xi}_{t},\bar{\xi}_{t} \rb}<\varepsilon$, then break.\\
 Update $\bar{\calX}_{t+1} = \bar{R}_{ \bar{\calX}_t}^Q\lb \bar{\xi}_t \rb$.
        }
\KwOut{$\bar{\calX}_{t}\in\overline{\calM}_{\vr}$.}
\end{algorithm}}
\subsubsection{Equivalence of Riemannian Gauss-Newton Methods under the Quotient and Embedded Geometries}
Recall that the set of fixed  tensor train rank  tensors forms a smooth embedded submanifold $\calM_{\vr}$ of dimension $\sum_{k= 1}^d r_{k-1}n_kr_k-\sum_{k = 1}^{d-1}r_k^2$. 
Let $T_{\calX}\calM_{\vr}$ be the tangent space of $\calM_{\vr}$ at $\calX$.
For an objective function $h\lb\calX\rb = \frac{1}{2}\fronorm{F\lb\calX\rb}^2$ defined on $\calM_{\vr}$, the Gauss-Newton direction $\xi\in T_{\calX}\calM_{\vr}$ is the solution of the following  Gauss-Newton equation~\cite[Section 8.4]{absil2009optimization}
\begin{align}\label{eq: GNe}
    \left\langle \mathrm{D} F\lb\calX\rb\lsb \eta\rsb,F\lb\calX\rb\right\rangle+ \left\langle \mathrm{D} F\lb\calX\rb\lsb \eta\rsb,\mathrm{D} F\lb\calX\rb\lsb \xi\rsb\right\rangle = 0,~\text{for all}~\eta\in T_{\calX}\calM_{\vr}.
\end{align}
The Riemannian Gauss-Newton algorithm under the embedded geometry (RGN (E)) is given in Algorithm~\ref{alg: RGN(E)}, where $R^E(\cdot)$ is a retraction from $ T_{\calX}\calM_{\vr}$ to $\calM_{\vr}$. A typical retraction is~\cite{steinlechner2016riemannian,uschmajew2020geometric}
\begin{align}\label{eq: retraction TTSVD}
R^E_{\calX}\lb\xi\rb = \TTSVD\lb \calX+\xi\rb,
\end{align}
where the TT-SVD can be computed efficiently by the TT-rounding procedure~\cite{oseledets2011tensor}. 
{\LinesNumberedHidden
\begin{algorithm}[ht!]
	\label{alg: RGN(E)}
	\caption{Riemannian Gauss-Newton under the embedded geometry (RGN (E))}
	\KwIn{Initial point $\calX\in\calM_{\vr}$, tolerance $\varepsilon>0$}
	\For{$t = 0,1,\cdots$}{ 
	Solving the Gauss-Newton equation~\eqref{eq: GNe} gives $\xi_t\in T_{\calX_t}\calM_{\vr}$.\\
 Check convergence: if $\fronorm{\xi_t}<\varepsilon$, then break.\\
 Update $\calX_{t+1} = R^E_{\calX_t} \lb\xi_t\rb$.
        }
\KwOut{$\calX_t\in\calM_{\vr}$.}
\end{algorithm}}

Next, we will present a new form of retraction which enables us to establish the equivalence of the Riemannian Gauss-Newton methods under the quotient and embedded geometries.  To this end, we first show that 
the map $\mathrm{D}\phi\lb\bar{\calX}\rb\vert_{\calH_{\bar{\calX}}}: \calH_{\bar{\calX}}\rightarrow T_{\phi\lb\bar{\calX}\rb}\calM_{\vr}$ is bijective.
\begin{lemma}\label{lem: Dphi bijective}
The mapping $\mathrm{D}\phi\lb\bar{\calX}\rb\vert_{\calH_{\bar{\calX}}}:\calH_{\bar{\calX}}\rightarrow T_{\phi\lb\bar{\calX}\rb}\calM_{\vr}$ is bijective.
\end{lemma}
\begin{proof}
  Suppose that there is a tangent vector $\bar{\xi}\in\calH_{\bar{\calX}}$ such that $\mathrm{D}\phi\lb\bar{\calX}\rb\lsb\bar{\xi}\rsb  = 0.$ By the chain rule and the relation $\phi = \Phi\circ\pi$, one has
\begin{align*}
\mathrm{D}\pi\lb\bar{\calX}\rb\lsb\bar{\xi}\rsb &= \mathrm{D}\lb\Phi^{-1}\circ\phi\rb\lb\bar{\calX}\rb\lsb\bar{\xi}\rsb\\
& = \mathrm{D}\Phi^{-1}\lb\phi\lb\bar{\calX}\rb\rb\lsb\mathrm{D}\phi\lb\bar{\calX}\rb\lsb\bar{\xi}\rsb\rsb\\
& = 0,
\end{align*}
which implies  that $\bar{\xi}\in\calV_{\bar{\calX}}$.  Since $\calV_{\bar{\calX}}$ is the  orthogonal complement of  $\calH_{\bar{\calX}}$,  $\bar{\xi}$ must be the zero element.  Thus, $\mathrm{D}\phi\lb\bar{\calX}\rb\vert_{\calH_{\bar{\calX}}}$ is injective.  

To show that $\mathrm{D}\phi\lb\bar{\calX}\rb\vert_{\calH_{\bar{\calX}}}$ is surjective, first note that  any $\xi\in T_{\phi\lb\bar{\calX}\rb}\calM_{\vr}$ can be expressed as~\cite[Section 9.3.4]{uschmajew2020geometric}
\begin{align*}
\xi = \sum_{k = 1}^d\phi\lb\calX^1,\cdots,\xi^k,\cdots,\calX^d\rb = \mathrm{D}\phi\lb\bar{\calX}\rb\lsb\bar{\xi}\rsb,
\end{align*}
where $\bar{\xi} = \left\lbrace\xi^1,\cdots,\xi^d\right\rbrace$.
In addition, one has
\begin{align*}
\mathrm{D}\phi\lb\bar{\calX}\rb\lsb\calP_{\bar{\calX}}^{\calH}\bar{\xi}\rsb &= \mathrm{D}\phi\lb\bar{\calX}\rb\lsb\calP_{\bar{\calX}}^{\calH}\bar{\xi}+\calP_{\bar{\calX}}^{\calV}\bar{\xi}\rsb =\mathrm{D}\phi\lb\bar{\calX}\rb\lsb\bar{\xi}\rsb= \xi.
\end{align*}
Therefore, for any $\xi\in T_{\phi\lb\bar{\calX}\rb}\calM_{\vr}$, there is at least one point $\calP_{\bar{\calX}}^{\calH}\bar{\xi}\in\calH_{\bar{\calX}}$ such that $\mathrm{D}\phi\lb\bar{\calX}\rb\lsb \calP_{\bar{\calX}}^{\calH}\bar{\xi} \rsb = \xi$.  
\end{proof}
One can easily verify that $\bar{\calX}\in \calH_{\bar{\calX}}$ and $\calX\in T_{\calX}\calM_{\vr}$, where the tangent space $T_{\calX}\calM_{\vr}$ is specified in Section~\ref{sec: implemention detail}. In addition, it is not hard to see that $$\mathrm{D}\phi\lb\bar{\calX}\rb\lsb \bar{\calX}\rsb = d\cdot \phi\lb\bar{\calX}\rb.$$
The bijective property of $\mathrm{D}\phi\lb\bar{\calX}\rb\mid_{\calH_{\bar{\calX}}}$ allows us to define the following  specific retraction  on the embedded submanifold $\calM_{\vr}$:
\begin{align}\label{eq: equivalent retraction}
    R^E_{\calX}\lb\xi\rb = \phi\lb \lb\mathrm{D}\phi\lb\bar{\calX}\rb\mid_{\calH_{\bar{\calX}}}\rb^{-1}\lb d\cdot\calX+\xi\rb\rb,
\end{align}
where  $\xi\in T_{\calX}\calM_{\vr}$ and $\calX = \phi\lb\bar{\calX}\rb =\phi\lb\left\lbrace \calX^1,\calX^2,\cdots,\calX^d\right\rbrace\rb$.
The following lemma shows that $R^E$ defined in~\eqref{eq: equivalent retraction} is indeed a retraction.
\begin{lemma}\label{lem: Re retraction}
    $R^E$ defined in~\eqref{eq: equivalent retraction} is  a retraction.
\end{lemma}
\begin{proof}
    First that  $R^E_{\calX}\lb 0\rb = \calX$ is evident. The bijective property of $\mathrm{D}\phi\lb\bar{\calX}\rb\mid_{\calH_{\bar{\calX}}}$ implies that  for any $\xi\in T_{\calX}\calM_{\vr}$, there is a unique $\bar{\xi} = \left\lbrace \xi^1,\xi^2,\cdots,\xi^d\right\rbrace \in\calH_{\bar{\calX}}$ satisfying $\mathrm{D}\phi\lb\bar{\calX}\rb\lsb \bar{\xi}\rsb = \xi$. Consequently,
    \begin{align*}
        \mathrm{D} R^E_{\calX}\lb0 \rb\lsb\xi\rsb &= \lim_{t\rightarrow 0}\frac{R^E_{\calX}\lb t\xi \rb- R^E_{\calX}\lb 0 \rb}{t}\\
        & = \lim_{t\rightarrow 0}\frac{\phi\lb \lb\mathrm{D}\phi\lb\bar{\calX}\rb\mid_{\calH_{\bar{\calX}}}\rb^{-1}\lb d\cdot\calX+t\xi\rb\rb- \phi\lb \lb\mathrm{D}\phi\lb\bar{\calX}\rb\mid_{\calH_{\bar{\calX}}}\rb^{-1}\lb d\cdot\calX\rb\rb}{t}\\
        & = \lim_{t\rightarrow 0}\frac{\phi\lb \bar{\calX}+t\bar{\xi}\rb - \phi\lb\bar{\calX}\rb}{t}\\
        & = \mathrm{D}\phi\lb\bar{\calX}\rb\lsb \bar{\xi}\rsb = \xi.
    \end{align*}
    Hence $\mathrm{D} R^E_{\calX}\lb0\rb\lsb\xi\rsb$ is an identity map.
\end{proof}
Now we are in position to establish the equivalence of the Riemannian Gauss-Newton methods under different geometries.
\begin{theorem}\label{thm: equivalent}
    Let $\bar{\xi}_{\bar{\calX}}\in\calH_{\bar\calX}$, $\xi\in T_{\calX}\calM_{\vr}$ be the solutions of the Gauss-Newton equations~\eqref{eq: GNQ equation} and~\eqref{eq: GNe} respectively, where  $\calX = \phi\lb\bar{\calX}\rb$. Then we have
    \begin{align*}
        \mathrm{D}\phi\lb\bar{\calX}\rb\lsb\bar{\xi}_{\bar{\calX}}\rsb = \xi.
    \end{align*}
Moreover, one has
\begin{align*}
    R^E_{\calX}\lb\xi\rb = \phi\lb \bar{R}^Q_{\bar{\calX}}\lb\bar{\xi}_{\bar{\calX}}\rb\rb,
\end{align*}
where the retractions $\bar{R}^Q\lb\cdot\rb$ and $R^E\lb\cdot\rb$ are defined in \eqref{eq: retraction total space} and~\eqref{eq: equivalent retraction}, respectively.
\end{theorem}
\begin{proof}
By the chain rule and the relation $\bar{F} = F\circ \phi$, the Gauss-Newton equation~\eqref{eq: GNQ equation} can be rewritten as
\begin{align} \label{eq: one more step RGNQ}
0 & = \left\langle   \mathrm{D} \bar{F}\lb\bar{\calX}\rb\lsb\bar{\eta}\rsb,\bar{F}\lb\bar{\calX}\rb \right\rangle + \left\langle \mathrm{D} \bar{F}\lb\bar{\calX}\rb\lsb \bar{\eta} \rsb, \mathrm{D} \bar{F}\lb\bar{\calX}\rb\lsb\bar{\xi}_{\bar{\calX}}\rsb \right\rangle\notag\\
& = \left\langle   \mathrm{D} F\lb\phi\lb\bar{\calX}\rb\rb\lsb\mathrm{D}\phi\lb\bar{\calX}\rb\lsb\bar{\eta}\rsb\rsb,\bar{F}\lb\bar{\calX}\rb \right\rangle + \left\langle  \mathrm{D} F\lb\phi\lb\bar{\calX}\rb\rb\lsb\mathrm{D}\phi\lb\bar{\calX}\rb\lsb \bar{\eta}\rsb\rsb, \mathrm{D} F\lb\phi\lb\bar{\calX}\rb\rb\lsb\mathrm{D}\phi\lb\bar{\calX}\rb\lsb\bar{\xi}_{\bar{\calX}}\rsb\rsb \right\rangle,
\end{align}
for all $\bar{\eta}\in\calH_{\bar{\calX}}$.
By the bijective property of $\mathrm{D}\phi\lb\bar{\calX}\rb\vert_{\calH_{\bar{\calX}}}:\calH_{\bar{\calX}}\rightarrow T_{\phi\lb\bar{\calX}\rb}\calM_{\vr}$, \eqref{eq: one more step RGNQ} is further equivalent to
 \begin{align}\label{eq: quotient gauss newton}
0  = \left\langle   \mathrm{D} F\lb\phi\lb\bar{\calX}\rb\rb\lsb\eta\rsb,\bar{F}\lb\bar{\calX}\rb \right\rangle + \left\langle  \mathrm{D} F\lb\phi\lb\bar{\calX}\rb\rb\lsb\eta\rsb, \mathrm{D} F\lb\phi\lb\bar{\calX}\rb\rb\lsb\xi\rsb \right\rangle,~\text{for all}~\eta\in T_{\phi\lb\bar{\calX}\rb}\calM_{\vr},
\end{align}
where $\mathrm{D}\phi\lb\bar{\calX}\rb\lsb \bar{\eta}\rsb = \eta$ and $\mathrm{D}\phi\lb\bar{\calX}\rb\lsb\bar{\xi}_{\bar{\calX}}\rsb = \xi$. This is indeed the same as  the Gauss-Newton equation in~\eqref{eq: GNe}. Moreover, with the retractions defined in~\eqref{eq: retraction total space} and~\eqref{eq: equivalent retraction},  one can easily verify that
\begin{align*}
    R^E_{\calX}\lb\xi\rb = \phi\lb \left\lbrace \calX^1+\xi^1,\cdots,\calX^d+\xi^d\right\rbrace\rb = \phi\lb \bar{R}^Q_{\bar{\calX}}\lb\bar{\xi}_{\bar{\calX}}\rb\rb,
\end{align*}
where $\bar{\xi}_{\bar{\calX}} = \left\lbrace\xi^1,\cdots,\xi^d\right\rbrace$.
\end{proof}
\begin{remark}
Note that the Riemannian Gauss-Newton search direction only depends on the differential of $F$ and  is independent of the Riemannian metric. This is a key property that underlies the equivalence of the Riemannian Gauss-Newton methods under the two geometries.
\end{remark}
\subsubsection{Computational Details}\label{sec: implemention detail}
For the low rank tensor completion problem in the tensor train format, the function $F\lb\calX\rb$ is given by $F\lb\calX\rb = \calP_{\Omega}\lb\calX\rb- \calP_{\Omega}\lb\calT\rb$.
 Notice that the update direction $\xi$ is also the solution of the following least squares problem~\cite[Section 8.4]{absil2009optimization}:
\begin{align}\label{eq: tensor least square}
    \xi = \argmin_{\xi\in T_{\calX}\calM_{\vr}}~\fronorm{ \calP_{\Omega}\calP_{T_{\calX}\calM_{\vr}}\xi+\calP_{\Omega}\lb\calX\rb-\calP_{\Omega}\lb\calT\rb}^2.
\end{align}
Assume $\calX$ is represented in the TT format~\eqref{eq: tt decomposition} with left-orthogonal core tensors $\left\lbrace\calX^1,\cdots,\calX^d\right\rbrace$,
i.e., $L\lb \calX^k\rb^\tran L\lb \calX^k\rb = \mI_{r_k}$, for $k = 1,\cdots,d-1$. The tangent space of $\calM_{\vr}$ at $\calX$ is given by~\cite{holtz2012manifolds}
\begin{align}\label{eq: tangent space}
T_{\calX} \calM_{\vr}  = \left\lbrace\sum_{k = 1}^d\phi\lb\left\lbrace\calX^1,\cdots,\delta\calX^k,\cdots,\calX^d\right\rbrace\rb\bigg\vert~\delta\calX^k\in\R^{r_{k-1}\times n_k\times r_k},L\lb\delta\calX^k\rb^\tran L\lb\calX^k\rb = \bm{0},k = 1,\cdots,d-1\right\rbrace.
\end{align}
Given a tensor $\calZ\in\R^{n_1\times\cdots\times n_d}$,  the orthogonal projection of $\calZ$ onto $T_{\calX} \calM_{\vr} $ is~\cite{lubich2015time} 
\begin{align}\label{eq: opembedd}
\calP_{T_{\calX}\calM_{\vr}}\lb \calZ\rb  = \sum_{k = 1}^d \phi\lb\left\lbrace\calX^1,\cdots,\delta\calZ^k,\cdots,\calX^d\right\rbrace\rb,
\end{align}
where  the left unfolding of $\delta\calZ^k$  is given by
\begin{align*}
L\lb\delta\calZ^k\rb = \lb\mI_{n_kr_{k-1}}-L\lb\calX^k\rb L\lb\calX^k\rb^\tran\rb \lb  \mI_{n_k}\otimes {\calX^{\leq k-1}}^\tran\rb \calZ^{<k>}\calX^{\geq k+1}\lb\mR^{k+1}\rb^{-1}\quad\mbox{for $k = 1,\cdots,d-1$,}
\end{align*}
and $L\lb\delta\calZ^d\rb = \lb \mI_{n_d}\otimes {\calX^{\leq d-1}}^\tran\rb \calZ^{<d>} $. 

Solving the least squares problem~\eqref{eq: tensor least square} directly is computationally prohibitive since the size of $\xi$ is $\prod_{k = 1}^d n_k$ which grows exponentially in $d$. Fortunately, $\xi\in T_{\calX}\calM_{\vr}$  implies that the degree of freedom in it is $\sum_{k= 1}^d r_{k-1}n_kr_k-\sum_{k = 1}^{d-1}r_k^2$. Therefore, the problem~\eqref{eq: tensor least square} can be rewritten as a  least squares  problem with the number of parameters equal to $\sum_{k= 1}^d r_{k-1}n_kr_k-\sum_{k = 1}^{d-1}r_k^2$. 
To do so, we need another representation of the tangent space.
\begin{lemma}
    The tangent space ${T_{\calX}\calM_{\vr}}$ in~\eqref{eq: tangent space} has the following alternative form:
    \begin{align}\label{eq: new tangent space}
        T_{\calX}\calM_{\vr} &= \Bigg\{ \sum_{k = 1}^{d-1} \ten_{<k>}\lb\lb\mI_{n_k}\otimes\calX^{\leq k-1}\rb L\lb\calX^k\rb^{\perp}\mD^k{\mQ^{k+1}}^\tran\rb+ \phi\lb\left\lbrace\calX^1,\cdots, \mD^d\right\rbrace\rb \big| \notag\\
        &\qquad\qquad\qquad\qquad\qquad\qquad\mD^k\in\R^{r_k\times (n_kr_{k-1}-r_k)},k = 1,\cdots,d-1, \mD^d\in\R^{r_{d-1}\times n_d}
            \bigg\},
    \end{align}
    where $\ten_{<k>}\lb\cdot\rb$ is the $k$-th tensorization operator: $\R^{n_1\cdots n_k\times n_{k+1}\cdots  n_d}\rightarrow \R^{n_1\times n_2\cdots \times n_d}$,
    $L\lb\calX^k\rb^{\perp}\in\R^{n_kr_{k-1}\times n_kr_{k-1}-r_k}$ is the orthogonal complement matrix of $ L\lb \calX^k\rb$, and $\QR\lb{\calX^{\geq k+1}}\rb = \mQ^{k+1}\mS^{k+1}$ with ${\mQ^{k+1}}^\tran\mQ^{k +1} = \mI_{r_k}$. 
\end{lemma}
\begin{proof}
To establish the equivalence of the tangent spaces in~\eqref{eq: tangent space} and~\eqref{eq: new tangent space}, we need to show that there exists a  one-to-one correspondence between the elements in two spaces. Given $\delta\calX^k$, the $k$-th unfolding of the $k$-th element in~\eqref{eq: tangent space} is
\begin{align*}
\lb\phi\lb\left\lbrace\calX^1,\cdots,\delta\calX^k,\cdots,\calX^d\right\rbrace\rb\rb^{<k>} = \lb\mI_{n_k}\otimes\calX^{\leq k-1}\rb L\lb \delta\calX^k\rb {\calX^{\geq k+1}}^\tran.
\end{align*}
Since $L\lb\delta\calX^k\rb^\tran L\lb\calX^k\rb = 0$, we have $L\lb\delta\calX^k\rb = L\lb\calX^k\rb^{\perp}\mA^k$ for some $\mA^k\in\R^{ n_kr_{k-1}-r_k\times r_k}$, where $L\lb\calX^k\rb^{\perp}\in\R^{n_kr_{k-1}\times n_kr_{k-1}-r_k}$ is the orthogonal complement matrix of $ L\lb \calX^k\rb$. Let  ${\calX^{\geq k+1}} = \mQ^{k+1}\mS^{k+1}$ be the QR decomposition of ${\calX^{\geq k+1}}$ with ${\mQ^{k+1}}^\tran\mQ^{k +1} = \mI_{r_k}$ and $\mS^{k+1}\in\R^{r_k\times r_k}$. It follows that
\begin{align*}
\lb\phi\lb\left\lbrace\calX^1,\cdots,\delta\calX^k,\cdots,\calX^d\right\rbrace\rb\rb^{<k>}   = \lb\mI_{n_k}\otimes\calX^{\leq k-1}\rb L\lb\calX^k\rb^{\perp}\mA^k{\mS^{k+1}}^\tran{\mQ^{k+1}}^\tran.
\end{align*}
Thus, the $k$-th element in~\eqref{eq: new tangent space} with  $\mD^k = \mA^k{\mS^{k+1}}^\tran$ corresponds to the $k$-th element in~\eqref{eq: tangent space}. 

Given $\mD^k$, the $k$-th element in~\eqref{eq: new tangent space} is
\begin{align*}
    \ten_{<k>}\lb\lb\mI_{n_k}\otimes\calX^{\leq k-1}\rb L\lb\calX^k\rb^{\perp}\mD^k{\mQ^{k+1}}^\tran\rb & = \ten_{<k>}\lb\lb\mI_{n_k}\otimes\calX^{\leq k-1}\rb L\lb\calX^k\rb^{\perp}\mD^k{\mS^{k+1}}^{-\tran}{\mS^{k+1}}^\tran{\mQ^{k+1}}^\tran\rb\\
    & = \phi\lb \left\lbrace \calX^1,\cdots,L^{-1}\lb L\lb\calX^k\rb^{\perp}\mD^k{\mS^{k+1}}^{-\tran}\rb,\cdots,\calX^d\right\rbrace\rb.
\end{align*}
Hence, the $k$-th element in~\eqref{eq: tangent space} with $\delta\calX^k = L^{-1}\lb L\lb\calX^k\rb^{\perp}\mD^k{\mS^{k+1}}^{-\tran}\rb$ corresponds to the $k$-th element in~\eqref{eq: new tangent space}.
\end{proof}

Notice that for $k =1,\cdots,d-1$,  the $k$-th unfolding of the $k$-th element in~\eqref{eq: opembedd} can be expressed as 
\begin{align*}
&\lb\phi\lb\left\lbrace\calX^1,\cdots,\delta\calZ^k,\cdots,\calX^d\right\rbrace\rb\rb^{<k>} \\
&\quad = \lb  \mI_{n_k}\otimes {\calX^{\leq k-1}}\rb\lb\mI_{n_kr_{k-1}}-L\lb\calX^k\rb L\lb\calX^k\rb^\tran\rb \lb  \mI_{n_k}\otimes {\calX^{\leq k-1}}^\tran\rb \calZ^{<k>}\calX^{\geq k+1}\lb\mR^{k+1}\rb^{-1}{\calX^{\geq k+1}}^\tran\\
&\quad =  \lb  \mI_{n_k}\otimes {\calX^{\leq k-1}}\rb L\lb\calX^k\rb^{\perp}{L\lb\calX^k\rb^{\perp}}^\tran\lb  \mI_{n_k}\otimes {\calX^{\leq k-1}}^\tran\rb \calZ^{<k>}\mQ^{k+1}{\mQ^{k+1}}^\tran.
\end{align*}
It is not hard to see that $\calP_{T_{\calX}\calM_{\vr}}\lb\calZ\rb = \calA\calA^\ast\lb\calZ\rb$ where 
the operators $\calA:  \R^{r_1\times (n_1-r_1)}\times \R^{r_2\times (n_2r_1-r_2)}\times\cdots\times\R^{r_{d-1}\times n_d}\rightarrow T_{\calX}\calM_{\vr}$ and $\calA^\ast:  \R^{n_1\times\cdots\times n_d} \rightarrow \R^{r_1\times (n_1-r_1)}\times \R^{r_2\times (n_2r_1-r_2)}\times\cdots\times\R^{r_{d-1}\times n_d}$ are defined by
\begin{align*}
    \calA\lb\left\lbrace \mD^k\right\rbrace_{k = 1}^d\rb &= \sum_{k = 1}^{d-1} \ten_{<k>}\lb\lb\mI_{n_k}\otimes\calX^{\leq k-1}\rb L\lb\calX^k\rb^{\perp}\mD^k{\mQ^{k+1}}^\tran\rb+ \phi\lb\left\lbrace\calX^1,\cdots, \mD^d\right\rbrace\rb,\\
    \calA^\ast\lb \calZ\rb &= \left\lbrace \left\lbrace{L\lb\calX^k\rb^{\perp}}^\tran\lb  \mI_{n_k}\otimes {\calX^{\leq k-1}}^\tran\rb \calZ^{<k>}\mQ^{k+1}\right\rbrace_{k = 1}^{d-1}, L^{-1}\lb \lb \mI_{n_d}\otimes {\calX^{\leq d-1}}^\tran\rb\calZ^{<d>}\rb \right\rbrace.
\end{align*}
It  can be easily verify that $\calA^\ast$ is the adjoint  of $\calA$. For $\lb i_1, i_2,\cdots,i_d\rb\in\Omega$, the corresponding measurement tensor is 
\begin{align*}
    \calB_i = \ve_{i_1}\circ\ve_{i_2}\circ\cdots\circ\ve_{i_d},
\end{align*}
where $\ve_{i_k}$ is the $i_k$-th canonical basis vector and the element of $\calB_i\in\R^{n_1\times\cdots\times n_d}$ is defined by
\begin{align*}
\calB_i\lb j_1,j_2,\cdots,j_d\rb = \ve_{i_1}\lb j_1\rb\cdot\ve_{i_2}\lb j_2\rb\cdots \ve_{i_d}\lb j_d\rb.
\end{align*}
Then  the objective function in~\eqref{eq: tensor least square} can be rewritten as 
\begin{small}
\begin{align}\label{eq: unconstrain tensor least square}
    &\fronorm{ \calP_{\Omega}\calP_{T_{\calX}\calM_{\vr}}\xi+\calP_{\Omega}\lb\calX\rb-\calP_{\Omega}\lb\calT\rb}^2\notag\\
    &  = \sum_{i = 1}^{\vert\Omega\vert}\lb \left\langle \calB_i,\calA\calA^\ast\xi\right\rangle+\left\langle \calB_i,\calX-\calT\right\rangle\rb^2 = \sum_{i = 1}^{\vert\Omega\vert}\lb \left\langle \calA^\ast\calB_i,\calA^\ast\xi\right\rangle+\left\langle \calB_i,\calX-\calT\right\rangle\rb^2\notag\\
    & = \sum_{i = 1}^{\vert\Omega\vert}\lb \sum_{k = 1}^{d-1}\left\langle {L\lb\calX^k\rb^{\perp}}^\tran\lb  \mI_{n_k}\otimes {\calX^{\leq k-1}}^\tran\rb \calB_i^{<k>}\mQ^{k+1},\mD^k \right\rangle +\left\langle  \lb \mI_{n_d}\otimes {\calX^{\leq d-1}}^\tran\rb\calB_i^{<d>},L\lb\mD^d\rb\right\rangle+\left\langle\calB_i,\calX-\calT\right\rangle\rb^2,
\end{align}
\end{small}where $\left\lbrace \mD^k\right\rbrace_{k =1}^d =\calA^\ast\lb \xi\rb$. Clearly, this is an unconstrained least squares problem with variables $\left\lbrace \mD^k\right\rbrace_{k = 1}^d$ whose dimension is $\sum_{k= 1}^d r_{k-1}n_kr_k-\sum_{k = 1}^{d-1}r_k^2$.

After solving  this problem,  the solution of ~\eqref{eq: tensor least square} can be obtained via $\xi = \calA\lb\left\lbrace \mD^k\right\rbrace_{k = 1}^d \rb$, since $\xi\in T_{\calX}\calM_{\vr}$. More precisely, we have 
\begin{align*}
    \xi 
    & = \sum_{k = 1}^{d-1} \ten_{<k>}\lb\lb\mI_{n_k}\otimes\calX^{\leq k-1}\rb L\lb\calX^k\rb^{\perp}\mD^k{\mQ^{k+1}}^\tran\rb+ \phi\lb\left\lbrace\calX^1,\cdots, \mD^d\right\rbrace\rb\\
    & = \sum_{k = 1}^{d-1} \ten_{<k>}\lb\lb\mI_{n_k}\otimes\calX^{\leq k-1}\rb L\lb\calX^k\rb^{\perp}\mD^k{\mS^{k+1}}^{-\tran}{\mS^{k+1}}^{\tran}{\mQ^{k+1}}^\tran\rb+ \phi\lb\left\lbrace\calX^1,\cdots, \mD^d\right\rbrace\rb\\
    & = \sum_{k = 1}^{d-1} \phi\lb \left\lbrace\calX^1,\cdots, L^{-1}\lb L\lb\calX^k\rb^{\perp}\mD^k{\mS^{k+1}}^{-\tran}\rb,\cdots,\calX^d\right\rbrace\rb+ \phi\lb\left\lbrace\calX^1,\cdots, \mD^d\right\rbrace\rb\\
    & = \mathrm{D}\phi\lb \bar{\calX}\rb \lsb \left\lbrace \left\lbrace L^{-1}\lb L\lb\calX^k\rb^{\perp}\mD^k{\mS^{k+1}}^{-\tran}\rb \right\rbrace_{k = 1}^{d-1},\mD^d\right\rbrace\rsb.
\end{align*}
Given $\xi = \calA\lb \left\lbrace \mD^k\right\rbrace_{k = 1}^d\rb$, by Theorem~\ref{thm: equivalent}, the Gauss-Newton update $\bar{\xi}_{\bar{\calX}}\in\calH_{\bar{\calX}}$ in~\eqref{eq: GNQ equation}  under the quotient geometry can be obtained by orthogonal projection~\eqref{eq: horizontal projection}:
\begin{align}\label{eq: obtain rgn update quo}
    \bar{\xi}_{\bar{\calX}} = \calP_{\bar{\calX}}^{\calH}\lb \left\lbrace \left\lbrace L^{-1}\lb L\lb\calX^k\rb^{\perp}\mD^k{\mS^{k+1}}^{-\tran}\rb \right\rbrace_{k = 1}^{d-1},\mD^d\right\rbrace\rb.
\end{align}

If we use the retraction defined in~\eqref{eq: retraction total space}, the main computational complexity of RGN (Q) lies in constructing and  solving the problem~\eqref{eq: unconstrain tensor least square}. It requires $O\lb dn^2r^2\rb$ flops to calculate the matrices $L\lb\calX^k\rb^{\perp}$ via Householder transformation.  To avoid computing $\mQ^{k+1}$, we can rewrite $\mQ^{k+1}$ as $\calX^{\geq k+1}{\mS^{k+1}}^{-\tran}$. By~\eqref{eq: recursive relation}, the computation of $\left\lbrace\mS^k\right\rbrace_{k = 2}^d$ can be implemented recursively which costs $O\lb dnr^3 \rb$ flops, see Algorithm~\ref{alg: calculate S}.  The sparsity of $\calB_i$ implies that calculating $L\lb\calX^k\rb^{\perp}\lb  \mI_{n_k}\otimes {\calX^{\leq k-1}}^\tran\rb \calB_i^{<k>}\calX^{\geq k+1}$ costs $O\lb dr^2+n^2r^3\rb$ flops. Thus, the total cost
 needed for constructing~\eqref{eq: unconstrain tensor least square} is $O\lb d\vert\Omega\vert n^2r^3+d^2\vert\Omega\vert r^2\rb$ flops. Since solving ~\eqref{eq: unconstrain tensor least square} costs  $O\lb d^2\vert\Omega\vert n^2r^4\rb$ flops, the main per iteration cost of RGN (Q) is $O\lb d^2\vert\Omega\vert n^2r^4\rb$ flops. For RGN (E), if the retraction in~\eqref{eq: retraction TTSVD}  is used, the TT-rounding procedure requires $O(dnr^3)$ flops~\cite{oseledets2011tensor} to compute the TT-SVD. Thus, 
  the leading order  per iteration cost of RGN (E) is still $O\lb d^2\vert\Omega\vert n^2r^4\rb$ flops. 
 {\LinesNumberedHidden
\begin{algorithm}[ht!]
	\label{alg: calculate S}
	\caption{Computation of $\left\lbrace \mS^k\right\rbrace_{k = 2}^d$}
	\KwIn{ $\bar{\calX} = \left\lbrace\calX^1,\cdots,\calX^d\right\rbrace$ }
    Compute QR decomposition: $\lsb\mQ^{d},\mS^{d}\rsb = \QR\lb R\lb\calX^d\rb^\tran\rb$; \\
	\For{$k = d-1,\cdots,2$}{  
    Compute QR decomposition: $\lsb \mV,\mS^{k}\rsb = \QR\lb \lb\mS^{k+1}\otimes\mI_{n_{k}}\rb R\lb\calX^{k}\rb^\tran\rb$;
	}
 \KwOut{$\left\lbrace \mS^k\right\rbrace_{k = 2}^d$.}
\end{algorithm}}

\section{Numerical Experiments}
In this section, we evaluate the empirical performance of the proposed algorithms against  existing algorithms for the tensor completion problem in the TT format.  Other tested algorithms, including Riemannian gradient descent (RGD (E))~\cite{cai2022provable,wang2019tensor},  Riemannian conjugate gradient (RCG (E))~\cite{steinlechner2016riemannian}, Riemannian trust region with finite-difference Hessian approximation (FD-TR)~\cite{psenka2020second}, are all based on the embedded geometry and implemented in the toolbox Manopt~\cite{boumal2014manopt}. For a fair comparison,  the step size selection criterion of RGD (E) has been modified to backtracking line search with RBB initial step size. We first compare the recovery ability of the tested algorithms on random low-rank tensors in Section~\ref{sec: numer recovery ability}. Then the convergence performance of these algorithms are tested in Section~\ref{sec: iteration count}. Finally, we evaluate the reconstruction quality of the tested algorithms on function-related tensors in Section~\ref{sec: interpolation}. 

\subsection{Recovery Ability vs. Oversampling Ratio and Condition Number}\label{sec: numer recovery ability}
We  investigate the recovery ability of the tested algorithms under different oversampling ratios and condition numbers. The oversampling (OS) ratio is defined as the ratio of the number of samples to the dimension,
\begin{align*}
\OS = \frac{\vert\Omega\vert}{\dim\lb\calM_{\vr}\rb},
\end{align*}
where $\vert\Omega\vert$ is the number of sampled entries, and $\dim\lb\calM_{\vr}\rb = \sum_{k= 1}^d r_{k-1}n_kr_k-\sum_{k = 1}^{d-1}r_k^2$ is the degrees of freedom of an $n_1\times\cdots\times n_d$ tensor with TT rank $(1,r_1,\cdots,r_{d-1},1)$. The condition number for a tensor $\calX$ is a natural generalization of the condition number of a matrix which is defined as~\cite{cai2022provable}
\begin{align*}
\kappa\lb\calX\rb = \frac{\sigma_{\max}\lb \calX\rb}{\sigma_{\min}\lb \calX   \rb },
\end{align*}
where $\sigma_{\max}\lb\calX\rb$ and $\sigma_{\min}\lb\calX\rb$ are defined by
\begin{align*}
    \sigma_{\max}\lb\calX\rb &:= \max\left\lbrace \sigma_{\max}\lb\calX^{<1>}\rb,\sigma_{\max}\lb\calX^{<2>}\rb,\cdots,\sigma_{\max}\lb\calX^{<d-1>}\rb\right\rbrace,\\
    \sigma_{\min}\lb\calX\rb &:= \min\left\lbrace \sigma_{\min}\lb\calX^{<1>}\rb,\sigma_{\min}\lb\calX^{<2>}\rb,\cdots,\sigma_{\min}\lb\calX^{<d-1>}\rb\right\rbrace.
\end{align*}

We fix $d = 3$, $n_1 = n_2 = n_3 = n = 100$, $r_1 = r_2 = r = 5$. Tests are conducted for two different oversampling ratios: $\OS = \left\lbrace 4,8\right\rbrace$, and for four different condition numbers: randomly generated tensors with condition number about $1$ and $\kappa = \left\lbrace 25 ,50,100\right\rbrace$.  Only $\OS\times\dim\lb\calM_{\vr}\rb$ entries of the ground truth tensor $\calT$ are observed. The test tensor $\calT$ with condition number about $1$ is constructed by the TT format with each core tensor being a random Gaussian tensor of appropriate size. The test tensor $\calT$ with a fixed condition number is generated in the following way. The core tensor $\calT^1$ is a random orthonormal matrix of size $ n \times r$. The core tensor $\calT^2$ is constructed by the Tucker decomposition: $\calT^2 = \calS\times_1\mU\times_2\mV\times_3\mW$, where $\calS\in\R^{r\times r\times r}$ is a diagonal tensor with ones along the superdiagonal and $\mU\in\R^{r\times r}$,  $\mV\in\R^{n\times r}$, $\mW\in\R^{r\times r}$ are  random orthonormal matrices. The core tensor $\calT^3$  is given  by $\calT^3 = \mX\bm{\Sigma}{\mY}^\tran$,  where $\mX,\mY$ are two  orthonormal matrices  of size $r\times r$ and $n\times r$ respectively,  and the diagonal entries of the singular matrix $\bm{\Sigma}$ are  linearly distributed from $1$ to $1/\kappa$. It  can be easily verify that $\kappa\lb\calT\rb = \kappa\lb\phi\lb\left\lbrace \calT^1,\calT^2,\calT^3\right\rbrace\rb\rb = \kappa$.

We run each algorithm 100 times for every combination of oversampling ratio and condition number. Tested algorithms are terminated if the relative error $\frac{\fronorm{\calP_{\Omega}\lb\calX_t\rb-\calP_{\Omega}\lb\calT\rb}}{\fronorm{\calP_{\Omega}\lb\calT\rb}}$ falls below $10^{-4}$ or $250$ number of iterations are reached.  An algorithm is considered  to have successfully reconstructed a test tensor if the output tensor $\calX_t$ satisfies $\frac{\fronorm{\calX_t-\calT}}{\fronorm{\calT}}\leq 10^{-3}$. The rate of successful recovery for different algorithms against different oversampling ratios and condition numbers are listed in Table~\ref{tab:succ recover rate}. 
It can be observed  from the table that the Riemannian gradient descent algorithm under the quotient geometry achieves the best reconstruction guarantee among all the tested algorithms. For a low oversampling ratio, the recovery ability of RGD (E), RCG (E), RGN (E)\footnote{Note that RGN(E) refers to Algorithm~\ref{alg: RGN(E)} with the TT-SVD retraction.}, RGN (Q)  degrades severely when the condition number increases, while RGD (Q), RCG (Q), and FD-TR can still achieve good performance. In the high oversampling ratio case, RGD (Q) and  RCG (Q) can successfully reconstruct the underlying tensor with a probability close to $1$ even when the condition number is large.
\begin{table}[ht!]
\renewcommand\arraystretch{1.5}
\centering
\caption{Successful recovery rate table for RGD (E), RGD (Q), RCG (E), RCG (Q), RGN (E), RGN (Q), FD-TR over $100$ random problem instances for $\OS = \left\lbrace 4,8\right\rbrace$ and random  tensor with $\kappa\approx 1$ and  $\kappa = \left\lbrace 25, 50,100\right\rbrace$}
\label{tab:succ recover rate}
\vspace{0.2cm}
\begin{tabular}{|c|c|c|c|c|c|}
\hline 
 &        & random $(\kappa\approx 1)$ & $\kappa= 25$ & $\kappa= 50$ & $\kappa= 100$\\
 \hline
 \multirow{7}{*}{OS = 4}                     & RGD (E) & 0.87   & 0.39                       & 0.11                      & 0.02                           \\
& RGD (Q) & $\bm{1}$   & $\bm{0.78}$                      & $\bm{0.59}$                       & $\bm{0.42}$                        \\
& RCG (E) & 0.99   & 0.47                       & 0.14                       & 0.01                        \\
 & RCG (Q) & 0.98   & 0.71                       & 0.51                       & 0.32                        \\
 & RGN (E) & 0.99   & 0.30                       & 0.06                       & 0                           \\
 & RGN (Q) & 0.99   & 0.30                       & 0.08                       & 0.01                        \\
 & FD-TR  & 0.96   & 0.74                       & 0.59                       & 0.41                        \\
 \hline
 \multicolumn{1}{|c|}{\multirow{7}{*}{OS = 8}} & RGD (E) & 0.99      & 0.98                       & 0.96                       & 0.83                        \\
 \multicolumn{1}{|c|}{}                        & RGD (Q) & $\bm{1}$      & $\bm{1}$                          & $\bm{1}$                          & $\bm{0.99}$                        \\
 \multicolumn{1}{|c|}{}                        & RCG (E) & 1      & 1                          & 0.98                      & 0.90                      \\
 \multicolumn{1}{|c|}{}                        & RCG (Q) & 1      & 1                          & 0.98                       & 0.98                        \\
 \multicolumn{1}{|c|}{}                        & RGN (E) & 1      & 0.98                       & 0.84                       & 0.12                        \\
 \multicolumn{1}{|c|}{}                        & RGN (Q) & 1      & 0.98                       & 0.88                       & 0.54                        \\
 \multicolumn{1}{|c|}{}                        & FD-TR  & 1      & 1                          & 0.98                       & 0.90  \\
 \hline
\end{tabular}
\end{table}

\subsection{Iteration Count and Runtime}\label{sec: iteration count}
In this section,  we first  investigate the iteration  count and  runtime of all the tested algorithms under different oversampling ratios.  The
algorithms are tested with $d = 9$, $n_1 = n_2 = \cdots = n_9 =  5$, $\vr = (1,3,5,10,10,10,10,5,3,1)$, $\OS = \left\lbrace 15,20\right\rbrace$, and they are terminated whenever the relative error  falls below $10^{-10}$. 
Tests are first conducted on random Gaussian tensors generated by the same procedure as in Section~\ref{sec: numer recovery ability}. We plot the relative residual against the iteration count and runtime in Figure~\ref{fig:my_labely}. It can be seen that RGN (E), RGN (Q), and FD-TR achieve superlinear convergences, while the other tested algorithms converge at a linear rate. For a
 low oversampling ratio, first-order methods under the quotient geometry converge faster than their counterparts based on the embedded geometry. 
\begin{figure}[ht!]
    \centering
    \includegraphics[width=0.4\textwidth]{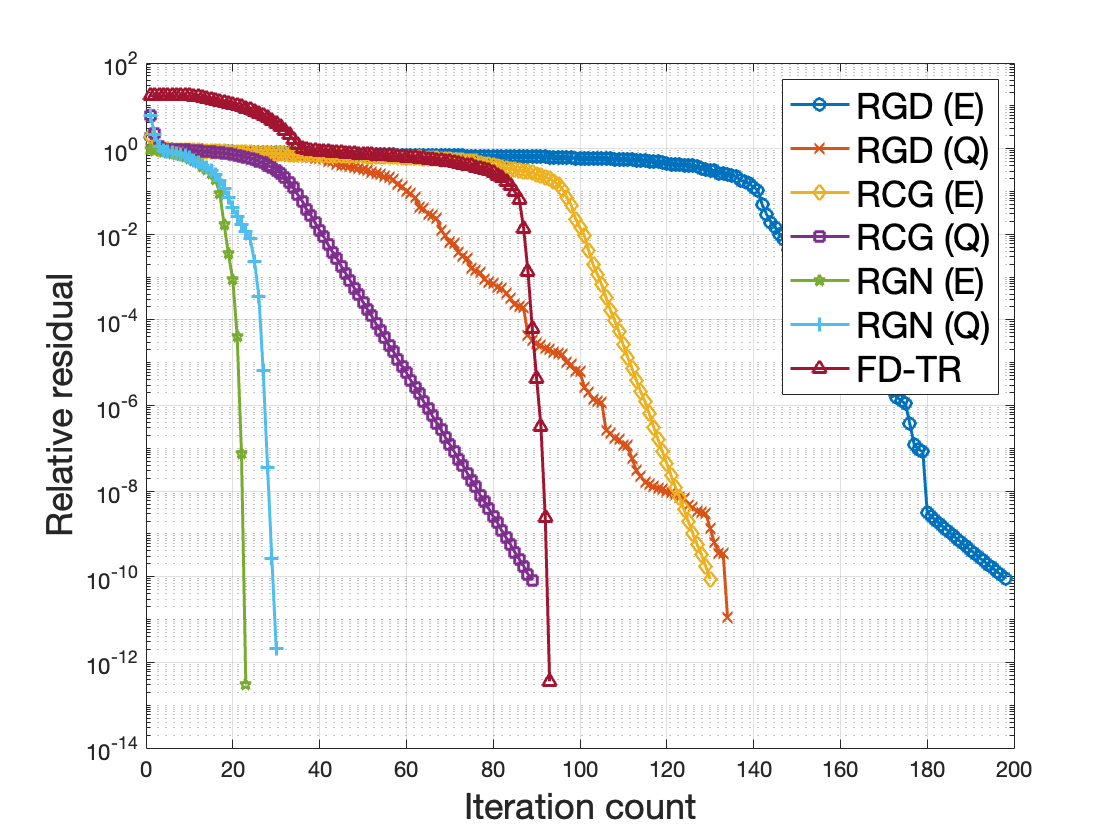}
    \includegraphics[width=0.4\textwidth]{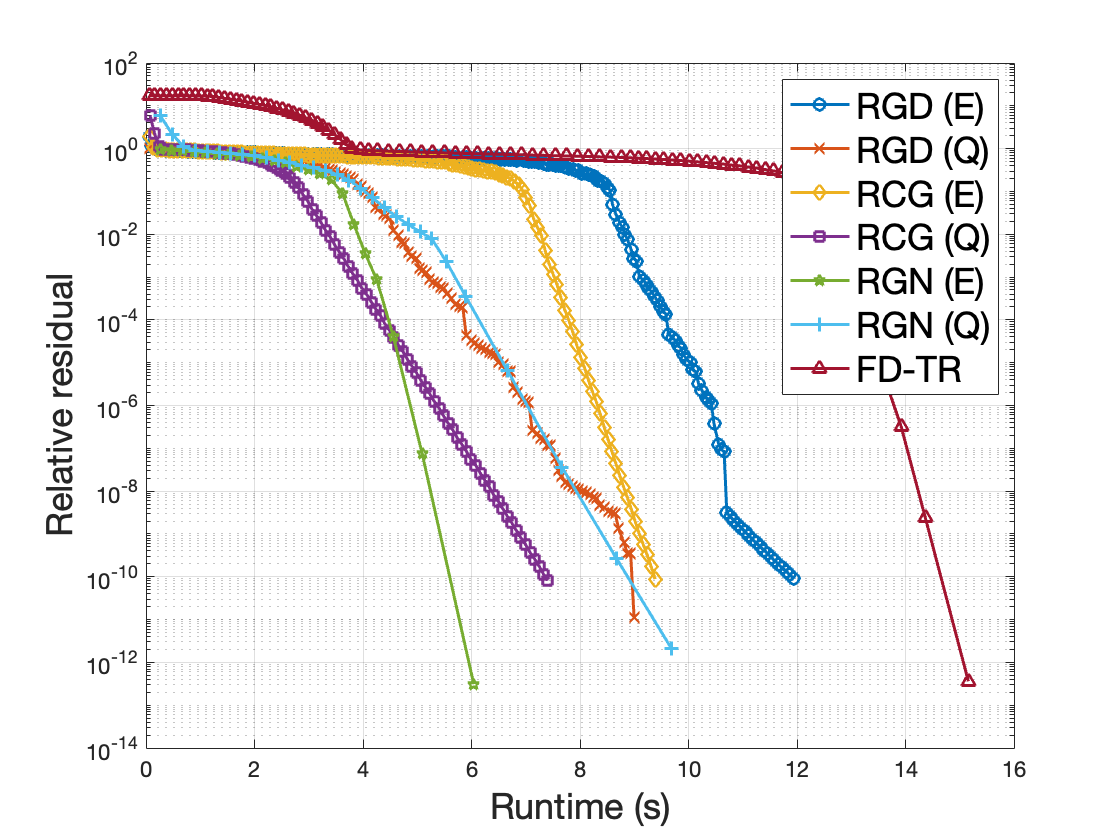}
    \includegraphics[width=0.4\textwidth]{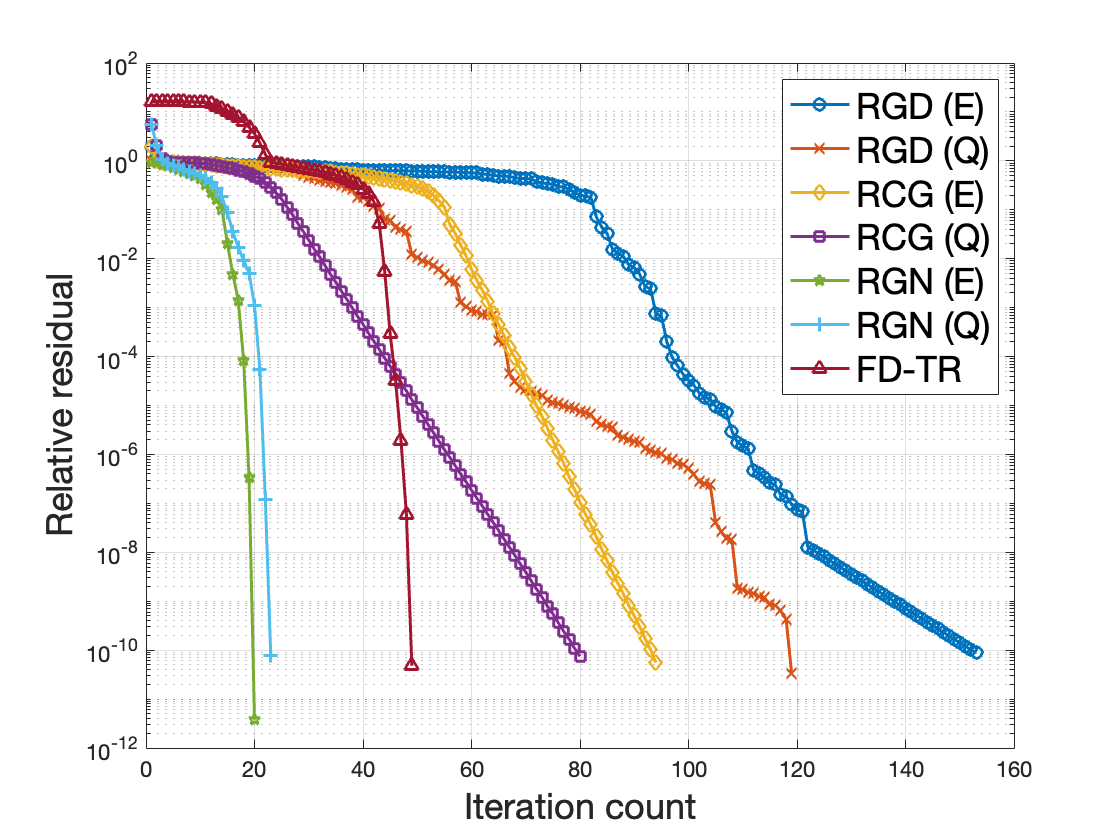}
    \includegraphics[width=0.4\textwidth]{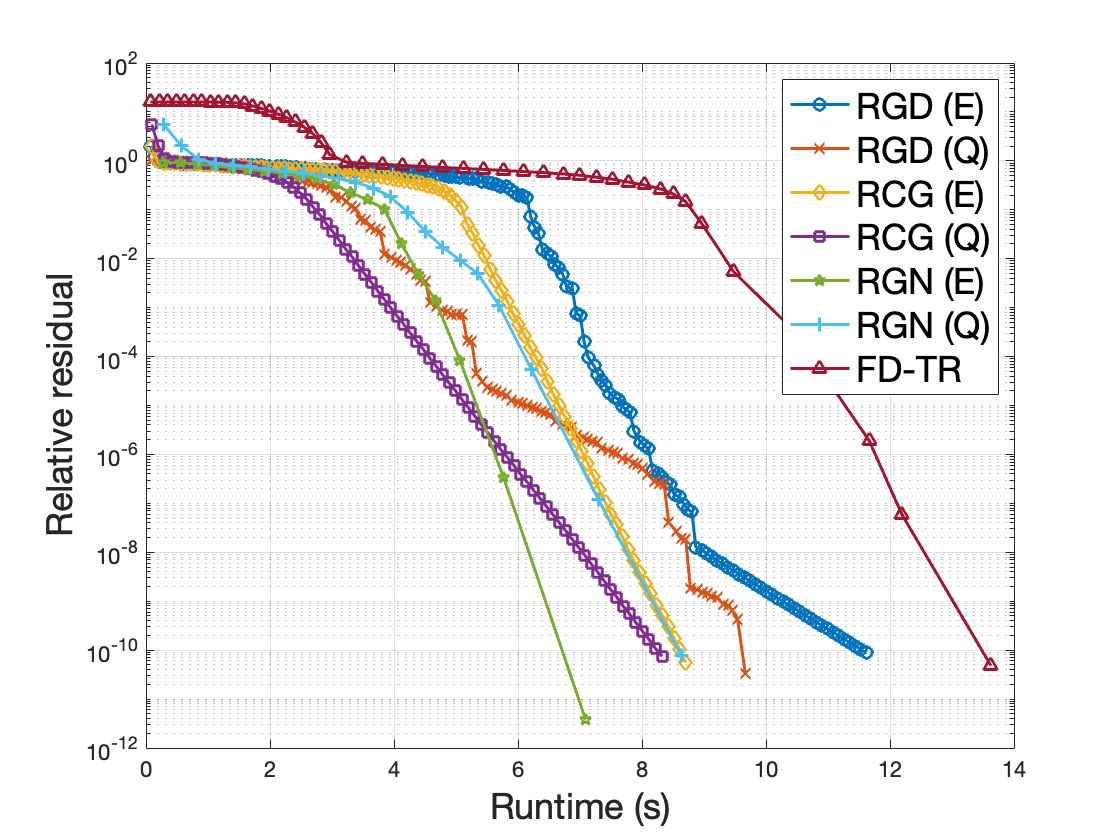}
    \caption{The relative residuals of tested algorithms with respect to the iteration count and
runtime (in seconds). Top panels: $\OS = 15$; bottom panels: $\OS = 20$. }
    \label{fig:my_labely}
\end{figure}

While  random tensors have benign condition numbers, we also compare the performance of the tested algorithms under a higher condition number $\kappa = 600$. We set 
 $d = 9$, $n_1 = n_2 = \cdots = n_9 =  5$, $\vr = (1,3,5,10,10,10,10,5,3,1)$, $\OS = \left\lbrace 15,20\right\rbrace$. The random tensors with fixed condition number are generated in the same way as in Section~\ref{sec: numer recovery ability}. Tested algorithms are terminated whenever the relative residual $\frac{\fronorm{\calP_{\Omega}\lb \calX_t\rb-\calP_{\Omega}\lb \calT\rb}}{\fronorm{\calP_{\Omega}\lb\calT\rb}}$ is less than $10^{-10}$ or $250$ number of iterations are reached.  In Figure~\ref{fig: cond}, we show  the relative residual of each algorithm  against the number of iteration and runtime. In this setting,  the algorithms  proposed in this paper are computationally more efficient than other state-of-the-art algorithms. Moreover,
 it can be observed that 
 RGD (Q) and RCG (Q) have a rapid initial residual decrease even when the condition number is large.  Thus, the total number of iterations for them to converge rely weakly on the condition number. In contrast,  the convergence of RGD (E) and RCG (E) relies heavily on the condition number of test tensors.  

 \begin{figure}[ht!]
    \centering
    \includegraphics[width=0.4\textwidth]{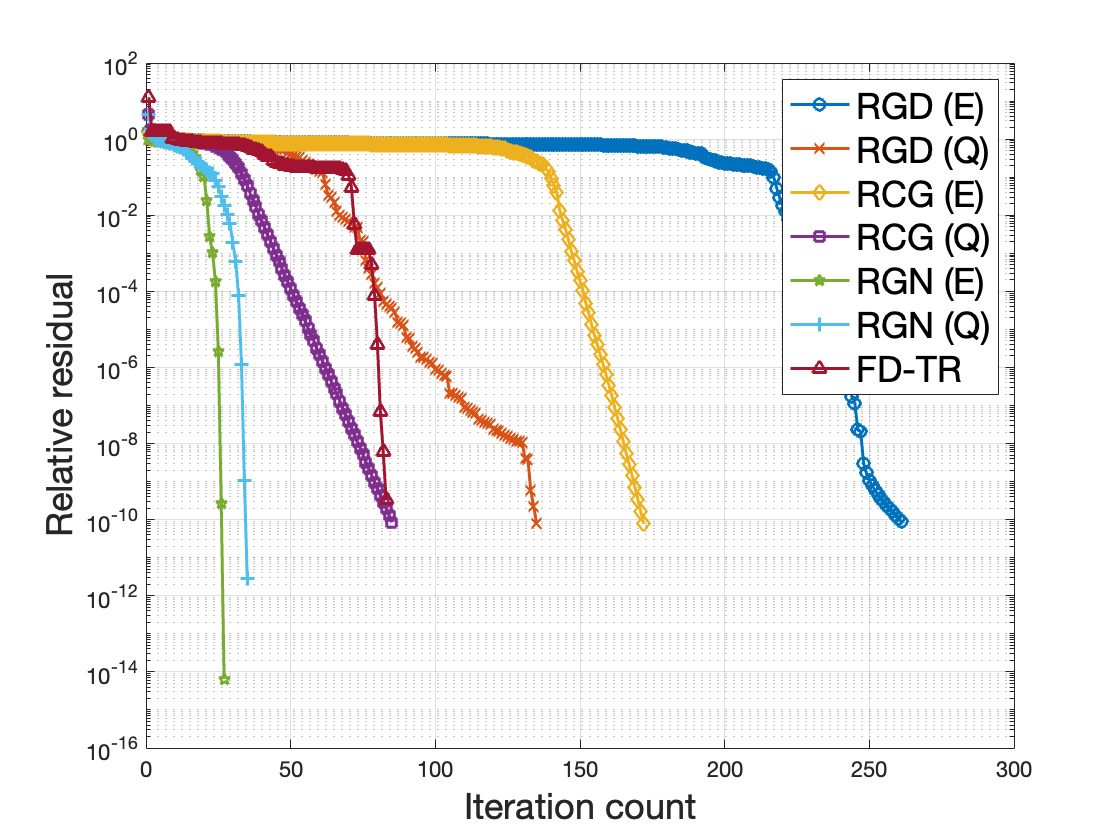}
    \includegraphics[width=0.4\textwidth]{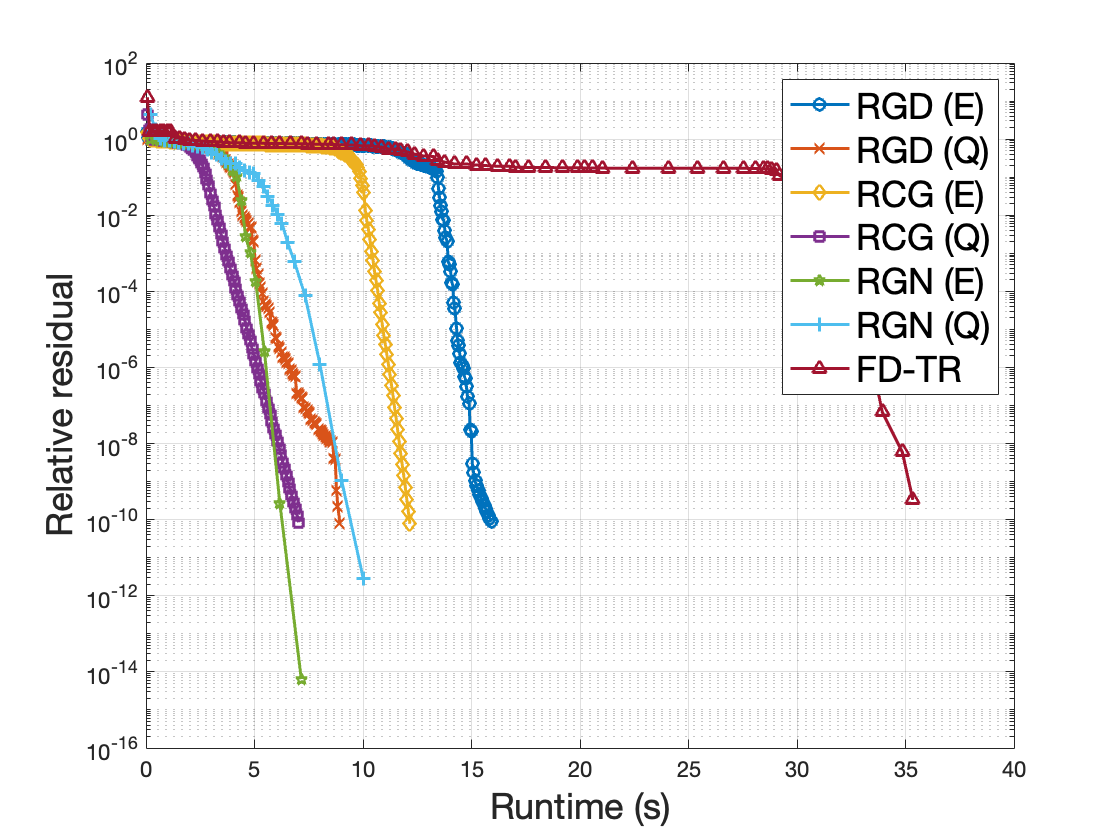}
    \includegraphics[width=0.4\textwidth]{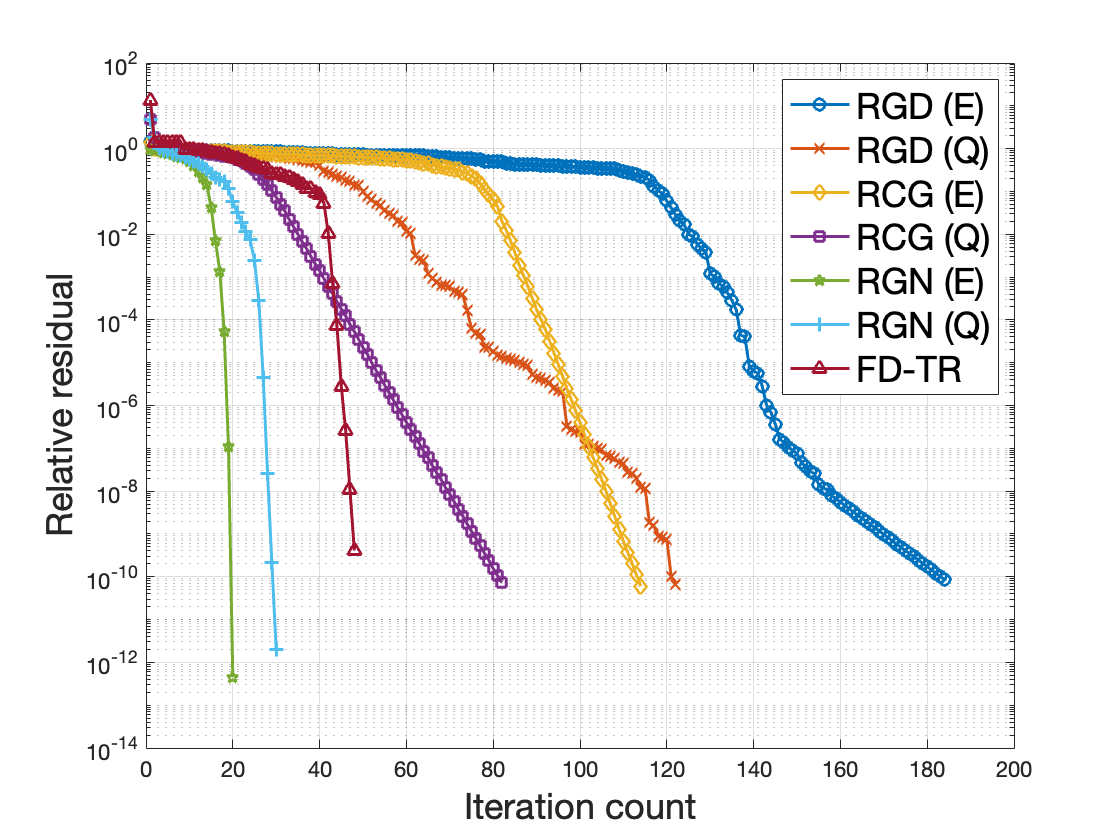}
    \includegraphics[width=0.4\textwidth]{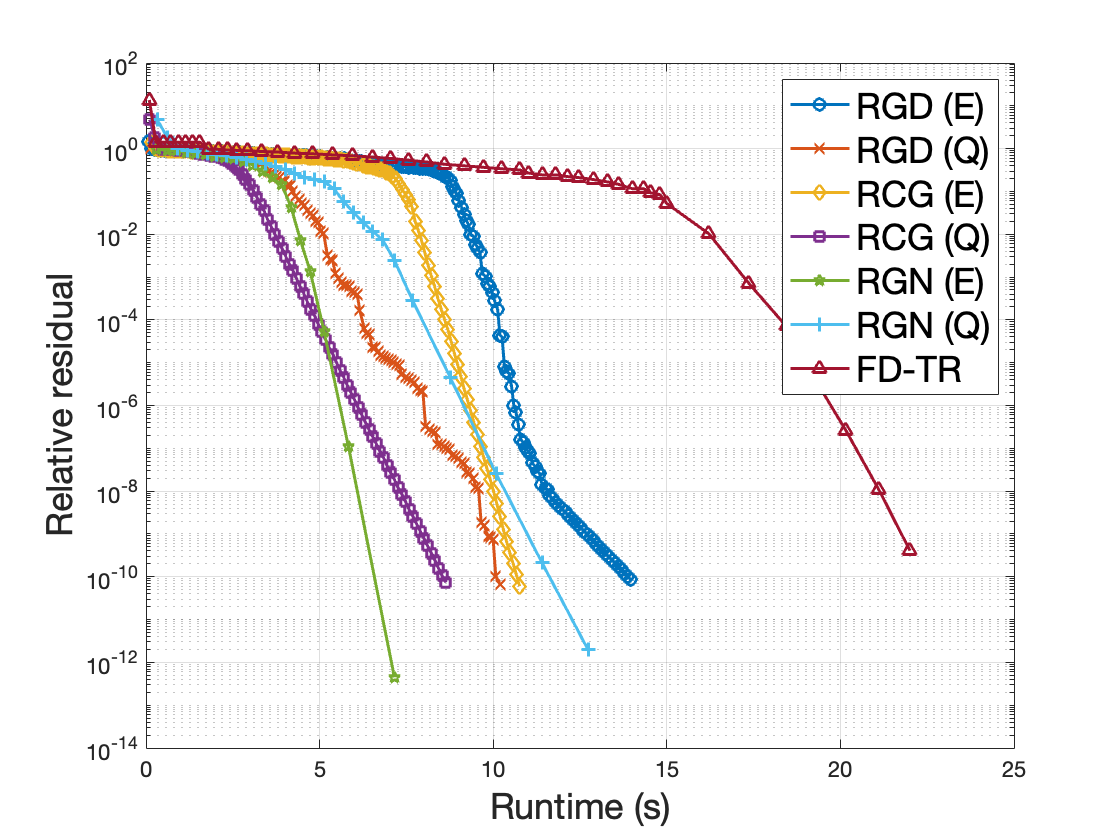}
    \caption{The relative residuals of tested algorithms with respect to the iteration count and
runtime (in seconds) under condition number $\kappa = 600$. Top panels: $\OS = 15$; bottom panels: $\OS = 20$. }
        \label{fig: cond}
\end{figure}

\subsection{Interpolation of High Dimensional Functions}\label{sec: interpolation}
Lastly, we  compare the performance of the tested algorithms on two discretization tensors of function data which were tested in~\cite{steinlechner2016riemannian}. The ground truth tensor $\calT_1\in\R^{n_1\times \cdots\times n_d}$ is constructed by 
\begin{align*}
    \calT_1\lb i_1,\cdots,i_d\rb = \exp\lb-\sqrt{\sum_{k = 1}^d \lb \frac{i_k}{n_k-1}\rb^2}\rb,
\end{align*}
while the other tensor $\calT_2\in\R^{n_1\times \cdots\times n_d}$ is generated by 
\begin{align*}
    \calT_2\lb i_1,\cdots,i_d\rb = \frac{1}{\sqrt{\sum_{k = 1}^d i_k^2}}.
\end{align*}
We set $d = 4$, $n_1 = n_2 = n_3 = n_4 = 20$. The above function-related tensors possess a property that the singular values of their unfolding formats delay sufficiently fast but do not become exactly zero. It implies that the underlying tensors are not  precisely low rank. To improve the reconstruction quality for these data, we  adopt
the rank-increasing strategy proposed in~\cite{steinlechner2016riemannian} for the tested algorithms.  For more details about the rank-increasing procedure, we refer the reader to~\cite[Section 4.9]{steinlechner2016riemannian}. The maximum rank is set to $\vr = \lb 1,5,5,5,1\rb$.  Tested algorithms are terminated at each rank-increasing step whenever the relative residual $\frac{\fronorm{\calP_{\Omega}\lb \calX_t\rb-\calP_{\Omega}\lb \calT\rb}}{\fronorm{\calP_{\Omega}\lb\calT\rb}}$ is less than $ 10^{-5}$, or $15$ ($20$ at final step) number of iterations are reached, or the relative change $\frac{\vert \fronorm{\calP_{\Omega}\lb \calX_t\rb-\calP_{\Omega}\lb \calT\rb}-\fronorm{\calP_{\Omega}\lb \calX_{t-1}\rb-\calP_{\Omega}\lb \calT\rb}\vert}{\fronorm{\calP_{\Omega}\lb \calX_{t-1}\rb-\calP_{\Omega}\lb \calT\rb}}$ is less than $ 10^{-3}$.

The computational results show that the second-order methods with the rank-increasing strategy
do not have a good performance on this task. Thus,  
we only present the     results for the first-order methods under different sizes of sampling set $\Omega$ in Table~\ref{tab: function data}. As in~\cite{steinlechner2016riemannian}, the relative test error on a randomly sampling set $\Gamma$ (outside of $\Omega$)  with $|\Gamma|=100$ is reported, which is defined by $\frac{\fronorm{\calP_{\Gamma}\lb \calX_t\rb- \calP_{\Gamma}\lb \calT\rb}}{\fronorm{ \calP_{\Gamma}\lb \calT\rb}}$. From the table, we can see that all the four algorithms achieve the overall similar performance. 


\begin{table}[ht!]
\renewcommand\arraystretch{1.5}
\centering
\caption{Comparison of the reconstruction quality of RGD (E), RGD (Q), RCG (E), RCG (Q) under different sizes of sampling set for underlying tensors $\calT_1$ and $\calT_2$.}
\label{tab: function data}
\vspace{0.2cm}
\begin{tabular}{|c|c|c|c|c|c|c|c|}
\hline
 &                & \multicolumn{3}{|c|}{underlying tensor $\calT_1$} & \multicolumn{3}{|c|}{underlying tensor $\calT_2$}\\
 \hline
& $\vert\Omega\vert/n^d$ & 0.001    & 0.01     & 0.1      & 0.001    & 0.01     & 0.1      \\
\hline
\multirow{3}{*}{RGD (E)} & relative error          & 9.13e-2  & 2.60e-3  & 1.21e-4  & 1.17e-0  & 8.89e-2  & 5.07e-4  \\
                         & runtime (s)          & 1.01     & 1.31     & 1.25     & 1.36     & 1.31     & 1.53     \\
                         & iteration count &122 &121 & 57 &140 & 124 &73\\
\hline
\multirow{3}{*}{RGD (Q)} & relative error          & 1.05e-1  & 2.60e-3  & 1.40e-4  & 2.16e-1  & 5.80e-2  & 5.33e-4  \\
                         & runtime  (s)          & 0.74     & 0.98     & 1.24     & 0.94     & 1.01     & 1.42     \\
                         & iteration count & 114 & 128 &67 &140 &136 &79\\
\hline
\multirow{3}{*}{RCG (E)} & relative error          & 9.60e-2  & 1.20e-3  & 8.09e-5  & 9.76e-1  & 7.23e-2  & 5.62e-4  \\
                         & runtime  (s)          & 0.55    & 0.89     & 1.04     & 0.65     & 0.93     & 1.24     \\
                         & iteration count & 99 &111 &45 &122 &124 &56\\
\hline
\multirow{3}{*}{RCG (Q)} & relative error          & 1.03e-1  & 1.20e-3  & 8.07e-5  & 1.80e-1  & 8.30e-3  & 5.57e-4  \\
                         & runtime  (s)          & 0.58     & 0.93     & 1.18     & 0.62    & 0.98     & 1.44 \\
                         & iteration count & 96&115 &51 &100 &127 &64\\
\hline

\end{tabular}
\end{table}

\section{Conclusion and Future Directions}
In this paper, we study the quotient geometry of the manifold of fixed  tensor train rank tensors under a preconditioned metric. Algorithms, including Riemannian gradient descent, Riemannian conjugate descent, and Riemannian Gauss-Newton,  have been proposed for the tensor completion problem based on the quotient geometry. It has been empirically demonstrated that the proposed algorithms are competitive with other existing algorithms on  random tensors as  well as  function-related tensors in terms of recovery ability, convergence performance, and reconstruction quality. 

There are a few lines of research for future directions. First, we would like to establish theoretical recovery guarantees of the proposed algorithms for the low tensor train rank tensor completion problem. Apart from that, it is also interesting to design efficient algorithms for the tensor robust principal component analysis (RPCA) problem under the quotient geometry based on the tensor  train format. Furthermore, it is likely to study the quotient geometry of the low-rank hierarchical Tucker tensors under the preconditioned metric studied in this paper.

\bibliographystyle{plain}
\bibliography{refTensor}
\end{document}